\documentclass[11pt]{article}
\usepackage{hyperref}\usepackage{authblk}
\usepackage{amsthm, amssymb, amsmath}
\usepackage[margin = 1in]{geometry} \usepackage{setspace}
\usepackage{enumitem} \setenumerate{label=(\alph*), leftmargin = *}
\usepackage{tikz}\usetikzlibrary{fit,positioning, patterns,
  shapes.geometric } \tikzset{mydot/.style={draw, circle, fill=black},
  myset/.style={draw, ellipse, very thick}}
\usepackage{mhchem}
\newcommand{\myleftrightarrows}[1]{\mathrel{\,\substack{\xleftarrow{#1}
      \\[-0.3ex] \xrightarrow{#1}}\,}}
\newcommand{\card}[1]{\left\lvert#1\right\rvert}
\newcommand{\im}{{\rm Im}} \newcommand{\cl}{{\sf Clopen}}
\newcommand{\ucl}{{\sf UpClopen}} \newcommand{\reg}{{\sf Reg}}
\newcommand{\N}{\mathbb N} \newcommand{\cB}{\mathcal B}
\newcommand{\cL}{\mathcal L} \newcommand{\cM}{\mathcal M}
\newcommand{\cN}{\mathcal N} \newcommand{\cP}{\mathcal P}
\newcommand{\cR}{\mathcal R} \newcommand{\A}{{\bf a}}
\newcommand{\x}{{\bf x}} \newcommand{\y}{{\bf y}}
\newcommand{\ii}{{\bf i}}
\theoremstyle{plain} \newtheorem{theorem}{Theorem}[section]
\newtheorem{corollary}[theorem]{Corollary}
\newtheorem{lemma}[theorem]{Lemma}
\newtheorem{proposition}[theorem]{Proposition}
\theoremstyle{definition} \newtheorem{definition}[theorem]{Definition}
\newtheorem{remark}[theorem]{Remark}
\newtheorem{example}[theorem]{Example}
\begin{document}
\title{Difference hierarchies and duality\\ with an application to
  formal languages\footnote{This project has received funding from the European Research
  Council (ERC) under the European Union's Horizon 2020 research and
  innovation program (grant agreement No. 670624).}}
  \date{}
\author[a]{C\' elia Borlido} \author[a]{Mai Gehrke} \author[b]{Andreas
  Krebs} \author[c]{Howard Straubing}

\affil[a]{\small Laboratoire J. A. Dieudonn\'e, CNRS, Universit\'e
  C\^ote d'Azur, France} \affil[b]{\small Wilhelm-Schickard-Institut
  Universit\"at T\"ubingen, Germany} \affil[c]{\small Boston College,
  United States}

\maketitle
\begin{abstract}
  The notion of a difference hierarchy, first introduced by Hausdorff,
  plays an important role in many areas of mathematics, logic and
  theoretical computer science such as descriptive set theory,
  complexity theory, and the theory of regular languages and automata.
  From a lattice theoretic point of view, the difference hierarchy
  over a bounded distributive lattice stratifies the Boolean algebra
  generated by it according to the minimum length of difference chains
  required to describe the Boolean elements.  While each Boolean
  element is given by a finite difference chain, there is no canonical
  such writing in general. We show that, relative to the filter
  completion, or equivalently, the lattice of closed upsets of the
  dual Priestley space, each Boolean element over the lattice has a
  canonical minimum length decomposition into a Hausdorff
  difference. As a corollary each Boolean element over a (co-)Heyting
  algebra has a canonical difference chain. With a further
  generalization of this result involving a directed family of
  adjunctions with meet-semilattices, we give an elementary proof of
  the fact that a regular language is given by a Boolean combination
  of purely universal sentences using arbitrary numerical predicates
  if and only if it is given by a Boolean combination of purely
  universal sentences using only regular numerical predicates.
\end{abstract}
\section{Introduction}\label{sec:intro}
Hausdorff introduced the notion of a difference hierarchy in his work
on set theory~\cite{Hausdorff57}. Subsequently, the notion has played
an important role in descriptive set theory as well as in complexity
theory. More recently, it has seen a number of applications in the
theory of regular languages and
automata~\cite{GlaserSchmitzSelivanov16,CartonPerrinPin17}.  From a
lattice theoretic point of view, the difference hierarchy over a
bounded distributive lattice $D$ stratifies the Booleanization, $B$,
of the lattice in question. The Booleanization of $D$ is the (unique
up to isomorphism) Boolean algebra containing $D$ as a bounded
sublattice and generated (as a Boolean algebra) by $D$. The
stratification is made according to the minimum length of difference
chains required to describe an element $b\in B$:

\begin{equation}
  b=a_1-(a_2-(\ldots
  (a_{n-1}-a_n)\hspace*{-1pt}.\hspace*{-1pt}.\hspace*{-1pt}.\hspace*{-1pt}))\label{eq:14}
\end{equation}

\noindent where $a_1\geq a_2\geq\ldots \geq a_{n-1}\geq a_n$ are elements of
$D$.
One difficulty in the study of difference hierarchies is that in
general elements $b\in B$ do not have canonical associated difference
chains.

Stone duality \cite{Stone37} represents any bounded distributive
lattice as the simultaneously compact and open subsets of an
associated topological space known as the Stone dual space of the
lattice. Priestley duality~\cite{Priestley70a} is a rephrasing of this
duality which uses the Stone space of the Booleanization equipped with
a partial order to represent the lattice as the closed and open upsets
of the associated Priestley space. Priestley duality provides an
elucidating tool for the study of difference hierarchies. For one, the
minimum length of difference chains for an element $b\in B$ has a nice
description relative to the Priestley dual space $X$ of $D$ as the
length of the longest chain of points $x_1\leq x_2\leq\dots\leq x_n$
in $X$ so that $x_i$ belongs to the clopen corresponding to $b$ if and
only if $i$ is odd. Further, if we allow difference chains of closed
upsets of the Priestley space, rather than clopen upsets, then every
element $b\in B$ has a canonical difference chain which is of minimum
length. In particular, if the lattice $D$ is a co-Heyting algebra,
then the canonical difference chain of closed upsets consists of
closed and open upsets and thus every $b\in B$ has a canonical
difference chain in $D$. We present this material, which is closely
related to work by Leo \`Esakia on skeletal subalgebras of closure
algebras~\cite{Esakia85}, in Section~\ref{sec:diff+closed}.

In Section~\ref{sec:dir-fam-adj}, we consider a situation where $B$ is
equipped with a directed family of adjunctions. Using Stone duality in
the form of canonical extensions, we generalize the results of
Section~\ref{sec:diff+closed}. In turn, the results of
Section~\ref{sec:dir-fam-adj} are used in Section~\ref{sec:LoW} for an
application in the setting of logic on words.  More precisely, we give
an elementary proof of the fact that a regular language is given by a
Boolean combination of purely universal sentences using arbitrary
numerical predicates if and only if it is given by a Boolean
combination of purely universal sentences using only regular
predicates. That is, expressed in a formula, we give an elementary
proof of the equality

\[
  \cB \Pi_1[\cN] \cap \reg = \cB\Pi_1[\reg].
\]

This result was first proved by Macial, P\'eladeau and Th\'erien
in~\cite{MacielPeladeauTherien00}. For more details,
see~\cite{straubing94}.

Before each of the main sections 3, 5, and 7, we include the
background needed.  Namely, in Section~\ref{sec:prel} we introduce the
basics on lattices and duality, Section~\ref{sec:canext} is an
introduction to canonical extensions, and Section~\ref{sec:back-LoW}
contains the preliminaries on recognition and logic on words.
\section{Preliminaries on lattices and duality}\label{sec:prel}
\paragraph{Semilattices, lattices, and Boolean algebras}

A \emph{semilattice} is an idempotent commutative monoid.  A
\emph{bounded distributive lattice} is a structure
$(D,\wedge,\vee,0,1)$ so that $(D,\wedge,1)$ and $(D,\vee,0)$ both are
semilattices and the operations~$\vee$ and~$\wedge$ distribute over
each other.
The relation $a\leq b$ if and only if
$a\wedge b=a$ if and only if $a\vee b=b$ is a partial order on $D$ and one
can recover the operations $\vee$ and $\wedge$  as the binary supremum 
and infimum, and the constants $1$ and $0$ as the maximum and minimum 
elements of the poset, respectively.
A \emph{Boolean algebra} is a bounded distributive lattice equipped
with a unary operation $\neg$ satisfying the identities $0 = a \wedge
\neg a$ and $1 = a \vee \neg a$.
The fundamental example of a Boolean algebra is a powerset with the
set-theoretic operations, obtained from the inclusion order and the
negation $\neg$ given by complementation. Thus any subset of a
powerset closed under the bounded lattice operations is a bounded
distributive lattice. We view the classes of bounded distributive
lattices and of Boolean algebras as categories in which the morphisms
are the algebraic homomorphisms, that is, the maps that preserve all
the basic operations.
\paragraph{Stone duality}
Stone duality~\cite{Stone37} shows that all bounded distributive
lattices are, up to isomorphism, of the above form. In fact, it does
more than that as it provides a category of topological spaces dually
equivalent to the category of bounded distributive lattices thus
yielding an embedding of each bounded distributive lattice into a
certain sublattice of the lattice of open sets of the corresponding
space. We work here with the equivalent formulation due to
Priestley~\cite{Priestley70a}, which uses ordered compact Hausdorff
spaces rather than the non-Hausdorff spaces of the classical Stone
duality.

Recall that a \emph{prime filter} of a lattice $D$ is a non-empty
upset $x$ of $D$ that is closed under binary meets and such that
whenever $a \vee b \in x$, we have $a \in x$ or $b \in x$.
The \emph{Priestley dual space} of~$D$ consists of the set $S(D)$ of
prime filters of $D$ equipped with the topology generated by the sets
$\widehat{a}=\{x\in S(D)\mid a\in x\}$ and their complements, where
$a\in D$. Further, this space is ordered by inclusion of prime
filters. One can show that the resulting ordered topological space
$(X,\pi,\leq)$ is compact and \emph{totally order disconnected} (TOD).
That is, if $x\nleq y$ in $X$ then there is a clopen upset $V$ of $X$
with $x\in V$ and $y\not\in V$. Totally order disconnected compact
spaces are called \emph{Priestley spaces} and the appropriate
structure preserving maps are the continuous and order preserving
maps. Indeed, given a bounded lattice homomorphism, one can show that
the preimage map restricted to prime filters is a continuous and order
preserving map between the corresponding Priestley spaces. In the
other direction, given a Priestley space $(X,\pi,\leq)$ the collection
$\ucl(X,\pi,\leq)$ of subsets of $X$ that are simultaneously upsets
and closed and open (called clopen) forms a lattice of sets. Further,
given a continuous and order preserving map between Priestley spaces,
the inverse image map restricted to the clopen upsets is a bounded
lattice homomorphism. These functors account for the dual equivalence
of the category of bounded distributive lattices and the category of
Priestley spaces. On objects, this means that $D\cong \ucl(S(D))$ (via
the map $a\mapsto\widehat{a}$) for any bounded distributive lattice
$D$ and $X\cong S(\ucl(X))$ (via the map $x\mapsto\{V\in \ucl(X)\mid
x\in V\}$) for any Priestley space $X$. In addition, the double dual
of a morphism, on either side of the duality, is naturally isomorphic
to the original. For more details
see~\cite[Chapter~11]{DaveyPriestley02}.
\paragraph{Booleanization}
The \emph{Booleanization}, or free Boolean extension, $D^-$, of a
bounded distributive lattice $D$ is a Boolean algebra with a bounded
lattice embedding $D\to D^-$, satisfying a universal property. Namely,
given any bounded lattice homomorphism $h\colon D\to B$ into a Boolean
algebra, there is a unique extension $h^-\colon D^-\to B$ so that the
following diagram commutes:
\begin{center}
  \begin{tikzpicture}[node distance = 20mm, semithick]
    \node (D) at (0,0) {$D$}; \node[right of = D] (D-) {$D^-$};
    \node[below of = D-, yshift = 5mm] (B) {$B$};
    \draw[->] (D) to (D-); \draw[->] (D) to node[left, yshift =
    -1mm]{\footnotesize$h$}(B); \draw[->, dashed] (D-) to node[right,
    yshift = 1mm] {\footnotesize$h^-$}(B);
  \end{tikzpicture}
\end{center}

\noindent The fact that Boolean algebras are bounded distributive
lattices with an additional operation satisfying some set of equations
implies, for general algebraic/categorical reasons, that a free
Boolean extension where $D\to D^-$ is only guaranteed to be a
homomorphism exists. The Booleanization has a few extra properties:
first of all the homomorphism $D\to D^-$ is in fact an
\emph{embedding}. Secondly, given \emph{any} embedding of $D$ into
\emph{any} Boolean algebra, the Boolean algebra generated by the image
is isomorphic to $D^-$. The former property follows as soon as one
shows the existence of at least one embedding of $D$ into some Boolean
algebra. The latter property follows as one can show that a
homomorphism of Boolean algebras which is injective on a distributive
lattice that generates the domain must be injective.

It is well known that every distributive lattice can be embedded in a
Boolean algebra. For one, it is a consequence of Stone duality since
the Stone map $a\mapsto \widehat{a}$ is a lattice embedding into the
powerset of the dual space. However, showing that the Stone map is an
embedding requires a non-constructive principle, so one may ask
whether such an embedding is available in a constructive manner. A
first attempt was made by MacNeille~\cite{MacNeille39}, but there was
a gap in his proof. A corrected version of MacNeille's argument was
subsequently provided by Peremans~\cite{Peremans57}. Later yet,
Gr\"atzer and Schmidt~\cite{GratzerSchmidt58} and then
Chen~\cite{Chen66} also provided such constructive embeddings.  In
particular, Gr\"atzer and Schmidt provide a simple proof by showing
that any bounded distributive lattice $D$ embeds in the Boolean
algebra of finitely generated congruences of $D$.

In the setting of Priestley duality, we have seen that $D$ is
isomorphic to the lattice of clopen upsets of its dual space $X$. It
thus follows that $D^-$ is isomorphic to the Boolean subalgebra of
$\cP(X)$ generated by $\ucl(X)$. One can show that this is the Boolean
algebra of all clopen subsets of $X$. That is,

\[
  D^-\cong\cl(X).
\]
In fact, if $(X,\pi,\leq)$ is the Priestley space of $D$, then
$(X,\pi,=)$ is the Priestley space of $D^-$ (and $(X,\pi)$ is the
Stone space of $D^-$). This is a consequence of the fact, due to
Nerode~\cite{Nerode59}, that each prime filter $x$ of $D$ extends to a
unique prime filter $x^-$ of $D^-$ given by

\[
  x^-=\{b\in D^-\mid\exists a,a'\in D\text{ with }a\in x\text{ and
  }a'\not\in x\text{ and }a- a'\leq b\},
\]

\noindent where $a- a'$ is shorthand for $a\wedge\neg
a'$. Accordingly, given an element $b\in D^-$ we write $\widehat{b}$
for the corresponding clopen of $X$. As follows from Nerode's result,
in terms of the prime filters of $D$, we have $\widehat{b}=\{x \in
S(D)\mid\exists a,a'\in D\text{ with } a\in x, a'\not\in x,\text{ and
} a- a'\leq b\}$.
\paragraph{Adjunctions}
Let $P$ and $Q$ be posets. Given maps $f:P\to Q$ and $g:Q\to P$, we
say that $(f,g)$ is an \emph{adjoint pair}\footnote{Sometimes called a
  covariant Galois connection.} provided

\[
  \forall p\in P\quad \forall q\in Q \qquad\left(\ f(p)\leq
    q\quad\iff\quad p\leq g(q)\ \right).
\]

\noindent Note that in this case, $f$ and $g$ uniquely determine each
other since

\[
  f(p)=\bigwedge\{q\in Q\mid p\leq g(q)\}\quad\text{and}\quad
  g(q)=\bigvee\{p\in P\mid f(p)\leq q\}.
\]
Accordingly, we also call $f$ the \emph{lower adjoint} of $g$ and $g$
the \emph{upper adjoint} of $f$. Thus a function has a lower,
respectively upper, adjoint provided the appropriate minima,
respectively maxima, exist. Also, one can show that lower adjoints
preserve all existing suprema, while upper adjoints preserve all
existing infima. In the case that the posets $P$ and $Q$ are complete
lattices this gives a simple criterion for the existence of adjoints.

\begin{proposition}\label{prop:adj}
  A map between complete lattices has a lower adjoint if and only if
  it preserves arbitrary meets, and it has an upper adjoint if and
  only if it preserves arbitrary joins.
\end{proposition}

Finally, we remark that if $(f,g)$ is an adjoint pair, then one can
show that $fgf=f$ and $gfg=g$ and that the operation $c=gf$ is always
a closure operator on $P$ in the sense that $p\leq c(p)=cc(p)$ holds
for all $p\in P$.  For the reader who would like to see the proofs of
these facts we refer to \cite[Chapter~7]{DaveyPriestley02}.
\paragraph{Heyting and co-Heyting algebras}
\emph{Heyting algebras} are the algebras for intuitionistic
propositional logic in the same sense that Boolean algebras are the
algebras for classical propositional logic.  A Heyting algebra is a
bounded distributive lattice equipped with an additional binary
operation, $\to$, which models the intuitionistic implication. The
order dual notion (obtained by interchanging the meet and join as well
as the $0$ and $1$ of the lattice) is called a co-Heyting algebra. We
will focus on co-Heyting algebras here as this is more convenient for
the sequel, but any result about the one notion has a corresponding
order dual result about the other notion.

\begin{definition}
  A \emph{co-Heyting algebra} is an algebra $(D,\wedge,\vee,0,1,/)$
  such that $(D,\wedge,\vee,0,1)$ is a bounded distributive lattice
  and the operation $/$ is the lower adjoint of the operation~$\vee$,
  in the sense that $((\ )/b,b\vee(\ ))$ is an adjoint pair for each
  $b\in D$.  That is, the operation $/$ is uniquely determined by the
  following property:

  \[
    \forall a, b, c\in D\qquad\quad ( \ a/b\leq c \quad\iff\quad a\leq
    b\vee c\ ).
  \]
\end{definition}

\noindent As stated above, these are precisely the order duals of
Heyting algebras. That is, $(D,\wedge,\vee,0,1,/)$ is a co-Heyting
algebra if and only if $(D,\vee,\wedge,1,0,\to)$ is a Heyting algebra,
where $a\to b=b/a$.  Notice that the adjunction property connecting
$\vee$ and $/$ implies that the co-Heyting operation on a distributive
lattice, if it exists, is unique and is given by

\[
  a/b=\bigwedge\{c\mid a\leq b\vee c\}.
\]

\noindent Thus one may think of a co-Heyting algebra as a special kind
of distributive lattice, namely one for which $\bigwedge\{c\mid a\leq
b\vee c\}$ exists for all $a,b\in D$ and for which the thus defined
operation satisfies the adjunction property in the definition of
co-Heyting algebras. One can show that the class of co-Heyting
algebras forms a variety, but here we are rather interested in a
description of these algebras involving the Booleanization of the
underlying bounded distributive lattice. It is not difficult to see
that the following proposition states that a bounded distributive
lattice $D$ is a co-Heyting algebra if and only if the inclusion of
$D$ in its Booleanization has a lower adjoint. See
\cite[Proposition~3]{Gehrke14} for the corresponding fact for Heyting
algebras. We include a proof for self-containment.

\begin{proposition}\label{prop:ceiling}
  Let $D$ be a bounded distributive lattice. Then $D$ is a co-Heyting
  algebra if and only if there is a ceiling function

  \[
    D^-\longrightarrow D, \quad b\mapsto \lceil b
    \rceil=\bigwedge\{c\in D\mid b\leq c\}.
  \]
\end{proposition}

\begin{proof}
  Let $ D$ be a bounded distributive lattice and let $a,b,c\in
  D$. Consider these as elements, first of the co-Heyting algebra $D$
  and then as elements of the Boolean algebra $ D^-$. We have
 
  \[
    a/b\leq c \ \iff\ a\leq b\vee c \ \iff\ a-b \leq c.
  \]
  
  \noindent Accordingly, we see that $ D$ is a co-Heyting algebra if
  and only if, the ceiling function above is well defined for all
  elements of $ D^-$ of the form $a-b$ with $a,b\in D$. However, since
  the elements of $ D^-$ are given by finite joins of elements of this
  form and since the ceiling function preserves joins (being a lower
  adjoint of the inclusion of $D$ in $D^-$) we can easily verify that
  this is sufficient to show that the ceiling function is globally
  defined by extending by join. To this end, let
  $d=\bigvee_{i=1}^n(a_i- b_i)$ be an arbitrary element of $ D^-$ and
  let $c\in D$, then we have

\begin{align*}
  d\leq c\iff &\forall i\in\{1,\ldots,n\}\quad a_i- b_i\leq c\\
  \iff & \forall i\in\{1,\ldots,n\}\quad a_i/ b_i\leq c\\
  \iff & \bigvee_{i=1}^n a_i/  b_i \leq c
\end{align*}

\noindent so that $\lceil d\rceil:= \bigvee_{i=1}^n a_i/b_i$ yields
the desired ceiling function.
\end{proof}

Since Heyting and co-Heyting algebras may be seen as certain bounded
distributive lattices, it is useful to understand which Priestley
spaces they correspond to. This was worked out by
\`Esakia~\cite{Esakia74} independently of Priestley's work: a bounded
distributive lattice $ D$ is a Heyting algebra if and only if, for any
clopen $V\subseteq X$ of the Priestley dual of $ D$, the generated
downset ${\downarrow}V$ is again clopen. This of course is equivalent
to the order dual characterization of co-Heyting algebras but we
include a proof to illustrate the correspondence between the algebraic
and topological formulations.

\begin{theorem}[{\cite{Esakia74}}]\label{thm:co-esak}
  Let $ D$ be a bounded distributive lattice and $X$ its Priestley
  dual space. Then $ D$ admits a co-Heyting structure if and only if
  for each $V\subseteq X$ clopen, ${\uparrow}V$ is again clopen. When
  this is the case, the map $\lceil\ \rceil\colon D^-\to D$ is
  naturally isomorphic to the map on clopen subsets of $X$ given by
  $V\mapsto{\uparrow}V$.
\end{theorem} 

\begin{proof}
  This is a simple consequence of Proposition~\ref{prop:ceiling}. Note
  that by total order disconnectedness of $X$, for any closed (and
  thus compact) $K\subseteq X$, we have
 
  \begin{equation}
    {\uparrow}K=\bigcap\{W\subseteq X\mid K\subseteq W
    \text{ and $W$ clopen upset}\}. \label{eq:1}
  \end{equation}
  
  \noindent Now we see easily that for $V\subseteq X$ clopen, there is
  a least clopen upset (i.e., element of~$ D$) above~$V$ if and only
  if the collection $\{W\subseteq X\mid V\subseteq W \text{ and $W$
    clopen upset}\}$ has a minimum and this, in turn, happens if and
  only if ${\uparrow}V$ is clopen.
\end{proof}
\paragraph{Closed subsets and closed upsets in Priestley spaces}
We highlight a few useful facts about closed subsets and upsets in
Priestley spaces. First we note that closed subspaces of Priestley
spaces, that is, closed subsets equipped with the inherited topology
and order, are again Priestley spaces. This is because the TOD
property is inherited by any subspace of a TOD space, and compactness
is inherited by any closed subspace of a compact space.  In fact, if a
Priestley space $X$ is dual to a bounded distributive lattice $D$, it
is not difficult to see that the closed subspaces of $X$ correspond to
the lattice quotients of $D$,
cf.~\cite[Section~11.32]{DaveyPriestley02}.

The following fact about upsets of closed sets will be used
extensively in the sequel. For a subset $S \subseteq P$ of a poset
$P$, we use $\min(S)$ to denote the set of minimal elements of $S$.
\begin{proposition}
  \label{p:2}
  Let $X$ be a Priestley space and $K \subseteq X$ a closed subset.
  Then, ${\uparrow}K={\uparrow} \min(K)$ and this is a closed subset
  of $X$.
\end{proposition}

\begin{proof}
  As seen in\eqref{eq:1} in the proof of Theorem~\ref{thm:co-esak},
  ${\uparrow}K$ is a closed subset of $X$ whenever $K\subseteq X$ is
  closed. Now consider $X$ as the dual space of a bounded distributive
  lattice $D$. Then the points of~$X$ are the prime filters of
  $D$. Let $x$ be any element of $K$ and let $C$ be a maximal chain of
  prime filters contained in $K$ with $x\in C$. Since $C$ is a chain,
  it is easy to show that $x_0=\bigcap_{x\in C}x$ is again a prime
  filter. Also, if $W= \widehat{a}$ is any clopen upset of $X$ with
  $K\subseteq W$, then $a\in y$ for all $y\in K$ and thus $a\in x_0$.
  It follows that $x_0\in W$ for all clopen upsets $W$ of $X$ with
  $K\subseteq W$ and thus $x_0\in{\uparrow}K$. Now by maximality of
  $C$ it follows that $x_0\in\min(K)$ and $x_0\leq x$. Thus
  ${\uparrow}K = {\uparrow} \min(K)$.
\end{proof}
\section{The difference hierarchy and closed
  upsets}\label{sec:diff+closed}
Let $ D$ be a bounded distributive lattice and~$ D^-$ its
Booleanization. Since $D^-$ is generated by~$ D$ as a Boolean algebra
and because of the disjunctive normal form of Boolean expressions,
every element of~$ D^-$ may be written as a finite join of elements of
the form $a- b$ with $a,b\in D$.  A fact, that is well known but
somewhat harder to see is that every element of $D^-$ is of the form

\begin{equation}
  a_1 - (a_2 - (\ldots -(a_{n-1}-a_n)
  \hspace*{-1pt}.\hspace*{-1pt}.\hspace*{-1pt}.\hspace*{-1pt})), \label{eq:11}
\end{equation}

\noindent for some $a_1, \dots, a_n \in D$.  The usual proof is by
algebraic computation and is not particularly enlightening.  It is
also a consequence of our results here. We begin with a technical
observation which is straightforward to verify.

\begin{proposition}\label{l:2}
  Let $B$ be a Boolean algebra and let $a_1, \ldots, a_{2m}$ be a
  decreasing sequence of elements of $B$. Then, the following equality
  holds:
  \[
    a_1-(a_2- (\dots-(a_{2m-1}- a_{2m})
    \hspace*{-1pt}.\hspace*{-1pt}. \hspace*{-1pt}.\hspace*{-1pt})) =
    (a_1-a_2) \vee (a_3-(a_4- (\dots-(a_{2m-1}- a_{2m})\dots))
  \]
 
  \noindent  where the join is disjoint,  and by induction we obtain\\[-1ex]
  \[
    a_1-(a_2- (\dots-(a_{2m-1}- a_{2m})
    \hspace*{-1pt}.\hspace*{-1pt}. \hspace*{-1pt}.\hspace*{-1pt})) =
    \bigvee_{n = 1}^m (a_{2n-1}-a_{2n})
  \]

  \noindent where the joinands are pairwise disjoint.
\end{proposition}

One problem with difference chain decompositions of Boolean elements
over a bounded distributive lattice, which makes them difficult to
understand and work with, is that, in general, there is no `most
efficient' such decomposition. We give an example of a Boolean element
over a bounded distributive lattice that illustrates the problem of
non-canonicity of difference chains.

\begin{example}\label{sec:ex-dif}
  We specify the lattice via its Priestley space. Consider $X={\mathbb
    N}\cup\{x,y\}$ equipped with the topology of the one point
  compactification of the discrete topology on ${\mathbb
    N}\cup\{x\}$. That is, the frame of opens of $X$ is:
  
  \[
    \pi=\cP({\mathbb N}\cup\{x\})\cup\{C\cup\{y\}\mid C\subseteq
    {\mathbb N} \cup\{x\}\text{ is cofinite}\}.
  \]

  The space $(X,\pi)$ is compact since any open cover must contain a
  neighborhood of $y$ and it must be cofinite. Covering the remaining
  points only requires a finite number of opens.
  
  The order relation on $X$ is as depicted. That is, the only
  non-trivial order relation in $X$ is $x\leq y$.
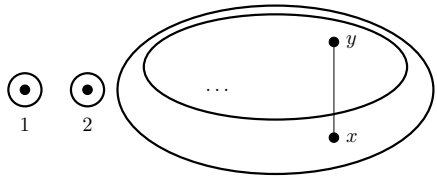
  \begin{center}
    \scalebox{.7}{\begin{tikzpicture}
        \node[scale=0.5,mydot,label={[label distance=.3cm]270:$1$}]
        (1) {}; \node[scale=0.5,mydot,label={[label
          distance=.3cm]270:$2$}] (2) [right=1cm of 1] {}; \node[]
        (c') [right = 1cm of 2]{}; \node[ ,label=center:$\dots$] (c)
        [right = 1cm of c']{}; \node[] (d) [right = 1cm of c]{};
        \node[scale=0.5,mydot,label=right:$y$] (y) [above right = 1cm
        of d]{}; \node[scale=0.5,mydot,label=right:$x$] (x) [below
        right = 1cm of d]{}; \node[myset, fit=(1)] {}; \node[myset,
        fit=(2)] {}; \node[myset, fit=(c)(x)(y), minimum height=3cm,
        minimum width=6cm] {}; \node[myset, fit=(c)(y), minimum
        height=2cm, minimum width=5cm] {}; \path[draw] (x)--(y);
      \end{tikzpicture}}
  \end{center}
  We argue that $(X,\pi,\leq)$ is TOD. Given $x_1,x_2\in X$ with
  $x_1\nleq x_2$, we have two cases. Either $x_2\in{\mathbb
    N}\cup\{x\}$ and then $\{x_2\}^c$ is a clopen upset containing
  $x_1$ but not $x_2$, or $x_2=y$ and then $x_1\in{\mathbb N}$ and
  $\{x_1\}$ is a clopen upset containing $x_1$ but not $x_2$. It
  follows that $X$ is a Priestley space.
  
  The clopen upsets of $X$ are the finite subsets of ${\mathbb N}$ and
  the cofinite subsets of $X$ containing $y$ and they form the lattice
  $D$ dual to $X$. Note that $V=\{x\}$ is clopen in $X$ and thus $V\in
  D^-$.  On the other hand, any clopen upset $W$ of $X$ containing $V$
  must be cofinite. We can write

  \[
    V=W-W'
  \]

  \noindent where $W'=W-\{x\}$ is also a clopen upset of $X$. It
  follows that there is no canonical choice for $W$ as
  ${\uparrow}V=\{x,y\}$ is not open and thus not in $D$.

  However, if we look for difference chains for $V$ relative to
  \emph{the lattice of closed subsets of }$X$, then we have a
  canonical choice of difference chain, namely $V=K_1-K_2$ where
  $K_1={\uparrow}V$ and $K_2=K_1-V$. This is a completely general
  phenomenon as we shall see next. Finally, we observe that since $V$
  is a clopen upset but ${\uparrow}V$ is not, $D$ is \emph{not} a
  co-Heyting algebra.
\end{example}

In this section we show that we may write each element of the
Booleanization of a bounded distributive lattice canonically as a
difference of \emph{closed upsets} (cf. Theorem~\ref{t:1}), and that,
in the case of a co-Heyting algebra, this provides a canonical
difference chain of the form~\eqref{eq:11} for each element of its
Booleanization (cf. Corollary~\ref{c:1}).

Recall that a subset $S\subseteq P$ of a poset is said to be
\emph{convex} provided $x\leq y\leq z$ with $x,z\in S$ implies $y\in
S$.

\begin{definition}\label{def:Hausdorff}
  If $P$ is a poset, $S\subseteq P$, and $p\in P$, then we say that
  $p_1<p_2<\dots<p_n$ in $P$ is an \emph{alternating sequence of
    length $n$ for $p$ (with respect to $S$)} provided
  \begin{enumerate}
  \item $p_i\in S$ for each $i\in\{1,\dots,n\}$ which is odd;
  \item $p_i\not\in S$ for each $i\in\{1,\dots,n\}$ which is even;
  \item $p_n=p$.
  \end{enumerate}
  Further, we say that $p\in P$ has \emph{degree} $n$ (with respect to
  $S$), written $\deg_{S}(p) = n$, provided~$n$ is the largest natural
  number $k$ for which there is an alternating sequence of length~$k$
  for~$p$.  In particular, if there is no alternating sequence for $p$
  with respect to $S$ (i.e. if $p\in P-{\uparrow}S$) then $\deg_S(p) =
  0$. Notice that an element of finite degree is of odd degree if and
  only if it belongs to $S$.  Also, if $S$ is convex, then every
  element of $S$ has degree~$1$, while every element of
  ${\uparrow}S-S$ has degree~$2$.
\end{definition}

In general, there will be non-empty subsets of posets with respect to
which no element has finite degree. However, that is not the case for
clopen subsets of Priestley spaces.

\begin{proposition}\label{prop:findeg}
  Let $X$ be a Priestley space and $V$ a clopen subset of $X$. Then
  every element of~$X$ has finite degree with respect to $V$. There
  are elements of degree $0$ if and only if $\,{\uparrow}V\neq X$, and
  the positive degrees achieved with respect to $V$ form an initial
  segment of the positive integers.
\end{proposition}

\begin{proof}
  Any element of the Booleanization of a bounded distributive lattice
  $D$ may be written as a finite disjunction of differences of
  elements from $D$ (using disjunctive normal form). Accordingly, if
  $V$ is a clopen subset of a Priestley space $X$, then there is an
  $m$ so that we may write

  \[
    V = \bigcup_{i = 1}^m (U_{i} - W_{i}),
  \]

  \noindent where $U_{i}, W_{i} \subseteq X$ are clopen upsets of
  $X$. In particular, since each $U_{i} - W_{i}$ is convex, by the
  Pigeonhole Principle, there is no alternating sequence with respect
  to $V$ of length strictly greater than $2m$, and thus, every element
  of $X$ has degree at most $2m$ with respect to $V$.

  Also, picking an alternating sequence for an element $x\in X$ whose
  length is the maximum possible, it is clear that the $k$-th element
  of such a sequence has degree $k$. Thus the set of positive degrees
  that are achieved form an initial segment of the positive
  integers. Finally, $x\in X$ has degree~$0$ if and only if
  $x\not\in{\uparrow}V$.
\end{proof}

\begin{lemma}\label{lem:deg1,2}
  Let $X$ be a Priestley space and $V$ a clopen subset of $X$. Define
  subsets of $X$ as follows:
  
  \[
    K_1 = {\uparrow}V, \quad K_2={\uparrow}({\uparrow}V-V).
  \]

  \noindent Then, for each $i\in\{1,2\}$, $K_i$ is closed and

  \[
    K_i=\{x\in X\mid \deg_V(x)\geq i\}={\uparrow}\{x\in X\mid
    \deg_V(x)=i\}.
  \]
\end{lemma}
\begin{proof}
  By Proposition~\ref{prop:findeg}, every element of $X$ has a finite
  degree. Also if $x\leq y$, then it is clear that
  $\deg_V(x)\leq\deg_V(y)$.  Furthermore, by Proposition~\ref{p:2}, we
  have that ${\uparrow}K= {\uparrow}\min(K)$ for any closed set
  $K$. Now since both $V$ and ${\uparrow}V-V$ are closed, it suffices
  to show that the elements of $\min(V)$ have degree $1$, and the
  elements of $\min({\uparrow}V-V)$ have degree $2$. Note that it is
  clear that $\deg_V(x)=1$ for any $x\in\min(V)$. Now suppose
  $x\in\min({\uparrow}V-V)$. Since $x\in {\uparrow}V$, there is $x'\in
  V$ with $x'\leq x$. Since $x\not\in V$, this is an alternating
  sequence of length $2$ for $x$. On the other hand, if $x_1< x_2
  <\dots < x_n=x$ is an alternating sequence for $x$, then
  $x_2\in{\uparrow}V-V$ and thus $x\not\in\min({\uparrow}V-V)$ unless
  $n=2$ and $x_2=x$. Thus $\deg_V(x)=2$ for any
  $x\in\min({\uparrow}V-V)$.
\end{proof}

\begin{corollary}\label{c:2}
  Let $X$ be a Priestley space, $V$ a clopen subset of $X$, and
  $G_1\supseteq G_2\supseteq\dots \supseteq G_{2p}$ a sequence of
  closed upsets in $X$ satisfying
 
  \begin{equation}
    V=G_1-(G_2-(\dots-(G_{2p-1}-G_{2p}){\dots})).\label{eq:15}
  \end{equation}
  
  \noindent If $K_1$ and $K_2$ are as defined in
  Lemma~\ref{lem:deg1,2}, then

  \[
    K_1\subseteq G_1,\quad K_2\subseteq G_2\ \text{ and }\
    G_1-G_2\subseteq K_1-K_2.
  \]
\end{corollary}
\begin{proof}
  By~\eqref{eq:15}, we have $V\subseteq G_1$. Also, since $G_1$ is an
  upset we have $K_1={\uparrow}V\subseteq G_1$. Now, since $G_1-G_2
  \subseteq V$ we have $G_1-V \subseteq G_2$ and as $G_2$ is an upset,
  it follows that ${\uparrow}(G_1-V) \subseteq G_2$. Also, $K_1
  \subseteq G_1$ implies $K_2 = {\uparrow}(K_1-V) \subseteq
  {\uparrow}(G_1-V)$ and thus, $K_2 \subseteq G_2$.  In particular, we
  have $G_1 - G_2 \subseteq V-K_2 \subseteq K_1 - K_2$.
\end{proof}

An iteration of Lemma~\ref{lem:deg1,2} leads to the main result of
this section.

\begin{theorem}\label{t:1}
  Let $X$ be a Priestley space and $V$ a clopen subset of $X$. Define
  a sequence of subsets of ${\uparrow}V$ as follows:
 
  \[K_1 = {\uparrow} V, \qquad K_{2i} = {\uparrow}(K_{2i-1} - V),
    \quad \text{and}\qquad K_{2i+1} = {\uparrow} (K_{2i} \cap V),\]
  
  \noindent for $i \ge 1$ (see Figure~\ref{fig:alt}).
  Then, $K_1\supseteq K_2\supseteq\dots$ is a decreasing sequence of
  closed upsets of $X$ and, for every $n \ge 1$, we have
  
  \begin{equation}
    K_n = \{x \in X \mid \deg_V(x) \ge n\} = {\uparrow}\{x \in X \mid
    \deg_V(x) = n\}.\label{eq:2}
  \end{equation}
  
  \noindent In particular,

  \begin{equation}
    V = \bigcup_{i = 1}^{m} (K_{2i-1}-K_{2i}) = K_1-(K_2-
    (\dots(K_{2m-1}- K_{2m})
    \hspace*{-1pt}.\hspace*{-1pt}. \hspace*{-1pt}.\hspace*{-1pt})),\label{eq:4}
  \end{equation}
  
  \noindent where $2m-1 = \max\{\deg_V(x) \mid x \in V\}$.
\end{theorem}
\begin{proof}
  First we note that if (\ref{eq:2}) holds, then (\ref{eq:4}) holds
  since $K_{2i-1}-K_{2i}$ will consist precisely of the elements of
  $V$ of degree $2i-1$ and since each element of $V$ has an odd degree
  less than or equal to the maximum degree achieved in $V$.  For the
  first statement and for (\ref{eq:2}), the proof proceeds by
  induction on the parameter $i$ used in (\ref{eq:4}). For $i=1$, we
  have $K_1=K_{2i-1}={\uparrow} V$ and $K_2=K_{2i}=
  {\uparrow}(K_{2i-1} - V)={\uparrow}({\uparrow} V - V)$ and thus
  $K_1\supseteq K_2$ are closed sets satisfying (\ref{eq:2}) by
  Lemma~\ref{lem:deg1,2}.

  For the inductive step, suppose the statements hold for $n\leq 2i$
  and notice that $K_{2i+1}={\uparrow} (K_{2i} \cap V)$ and
  $K_{2i+2}={\uparrow}(K_{2i+1} - V)$ are in fact the sets $K_1$ and
  $K_2$ of Lemma~\ref{lem:deg1,2} when we apply it to the Priestley
  space $X'=K_{2i}$ and its clopen subset $V'=K_{2i} \cap V$.
  Thus, to complete the proof, it suffices to notice that, for every
  $x \in X'$, we have $\deg_{V}(x) = \deg_{V'}(x) + 2n$.
\end{proof}
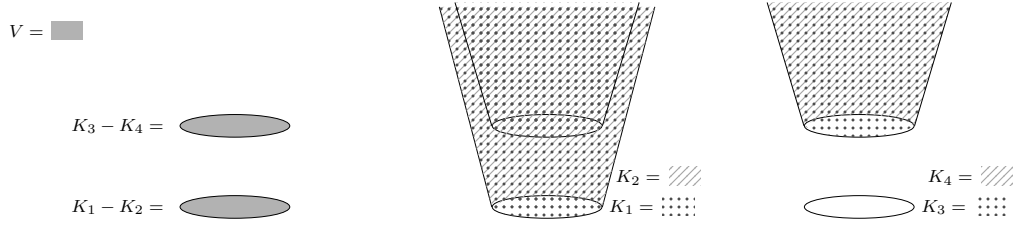
\begin{figure}
  \centering \scalebox{.83}{\begin{tikzpicture} \node[ellipse, draw =
      black,minimum width=50pt, fill = black!30] (1) at (0,0) {};
      \node[ellipse, draw = black,minimum width=50pt, fill = black!30,
      xshift = -10mm] (2) at (1,1.3) {};
      \node[left of = 1, xshift = -25pt]{\scriptsize$K_1-K_2 =$};
      \node[left of = 2, xshift = -25pt]{\scriptsize$K_3-K_4 =$};
      \node[ellipse, draw = black,minimum width=50pt] (11) at (5,0)
      {};
      \node[ellipse, draw = black,minimum width=50pt, pattern = dots,
      pattern color = black!70, xshift = -10mm] (12) at (6,1.3) {};
      \node[ellipse, minimum width=50pt, pattern = dots, pattern color
      = black!70] at (5,0) {};
      \node[above left of = 11, xshift = -30pt, yshift = 75pt] (inv1)
      {}; \node[above right of = 11, xshift = 30pt, yshift = 75pt]
      (inv2) {};
      \draw[black] (11.west) -- (inv1); \draw[black] (11.east) to
      (inv2);
      \node[above left of = 12, xshift = -23pt, yshift = 40pt] (inv3)
      {}; \node[above right of = 12, xshift = 23pt, yshift = 40pt]
      (inv4) {};
      \path[fill = black!30, pattern = north east lines, pattern color = black!35]
      (11.west) -- (inv1) -- (inv2) -- (11.east) -- (11.north east) --
      (11.north) -- (11.north west) -- (11.west) -- (inv1);
      \path[fill = black!30, pattern = dots, pattern color = black!70]
      (12.west) -- (inv3) -- (inv4) -- (12.east) -- (12.west) --
      (inv3);
      \path[fill = black!30, pattern = dots, pattern color = black!70]
      (11.west) -- (inv1) -- (inv2) -- (11.east) -- (11.west) --
      (inv1);
      \draw[black] (12.west) to (inv3); \draw[black] (12.east) to
      (inv4);
      \node[ellipse, draw = black,minimum width=50pt] (21) at (10,0)
      {}; \node[ellipse, draw = black,minimum width=50pt, xshift =
      -10mm] (22) at (11,1.3) {};
      \node[above left of = 21, xshift = -30pt, yshift = 75pt] (inv21)
      {}; \node[above right of = 21, xshift = 30pt, yshift = 75pt]
      (inv22) {};
      \node[above left of = 22, xshift = -23pt, yshift = 40pt] (inv23)
      {}; \node[above right of = 22, xshift = 23pt, yshift = 40pt]
      (inv24) {};
      \path[fill = black!30, pattern = north east lines, pattern color = black!35]
      (22.west) -- (inv23) -- (inv24) -- (22.east) -- (22.north east)
      -- (22.north) -- (22.north west) -- (22.west) -- (inv23);
      \path[fill = black!30, pattern = dots, pattern color = black!70]
      (22.west) -- (inv23) -- (inv24) -- (22.east) -- (22.south east)
      -- (22.south) -- (22.south west) -- (22.west) -- (inv23);
      \draw[black] (22.west) to (inv23); \draw[black] (22.east) to
      (inv24);
      \node[right of = 11,xshift = 10pt](K1){\scriptsize $K_1 = $};
      \node[right of = K1, rectangle, pattern = dots, pattern color =
      black!70, xshift = -7pt, yshift = .5pt, minimum height = 3mm,
      minimum width=5mm]{};
      \node[above of = K1,xshift = 3pt, yshift =
      -15pt](K2){\scriptsize $K_2 = $}; \node[right of = K2,
      rectangle, pattern = north east lines, pattern color = black!35,
      xshift = -7pt, yshift = .5pt, minimum height = 3mm, minimum
      width=5mm ]{};
      \node[right of = 21,xshift = 10pt](K3){\scriptsize $K_3 = $};
      \node[right of = K3, rectangle, pattern = dots, pattern color =
      black!70, xshift = -7pt, yshift = .5pt, minimum height = 3mm,
      minimum width=5mm]{};
      \node[above of = K3,xshift = 3pt, yshift =
      -15pt](K4){\scriptsize $K_4 = $}; \node[right of = K4,
      rectangle, pattern = north east lines, pattern color = black!35,
      xshift = -7pt, yshift = .5pt, minimum height = 3mm, minimum
      width=5mm]{};
      \node[above of = 2,xshift = -95pt, yshift = 15pt](V){\scriptsize
        $V = $}; \node[right of = V, rectangle, minimum height = 3mm,
      minimum width=5mm, xshift = -10pt, fill = black!30]{};
    \end{tikzpicture}}
  \caption{Graphical representation of the sequence defined in
    Theorem~\ref{t:1}.}
  \label{fig:alt}
\end{figure}

The corresponding generalization of Corollary~\ref{c:2} goes as
follows, and it proves the canonicity of writing~\eqref{eq:4}.

\begin{proposition}\label{p:5}
  Let $X$ be a Priestley space and $V\subseteq X$ be a clopen subset
  of $X$. Let $G_1 \supseteq G_2 \supseteq \dots \supseteq G_{2p}$ be
  a decreasing sequence of closed upsets of $X$ satisfying
 
  \begin{equation}
    \label{eq:18}
    V = \bigcup_{i = 1}^{p} (G_{2i-1}-G_{2i}) = G_1-(G_2-
    (\dots(G_{2p-1}- G_{2p})
    \hspace*{-1pt}.\hspace*{-1pt}. \hspace*{-1pt}.\hspace*{-1pt})).
  \end{equation}
  
  \noindent Then, taking $(K_i)_{i \ge 1}$ as defined in
  Theorem~\ref{t:1}, we have $p \ge m$ and, for every $n \in \{1,
  \dots, p\}$, the following inclusions hold:
 
  \begin{equation}
    K_{2n-1}\subseteq G_{2n-1},\ K_{2n} \subseteq G_{2n}, \
    \text{and} \ \bigcup_{i = 1}^n (G_{2i-1}-G_{2i}) \ \subseteq\
    \bigcup_{i = 1}^n (K_{2i-1} - K_{2i}).\label{eq:5}
  \end{equation}
  
\end{proposition}
\begin{proof}
  We proceed by induction on $n$. The case $n = 1$ is the content of
  Corollary~\ref{c:2}. Now suppose that (\ref{eq:5}) holds for a
  certain $n \in\{1, \dots, p\}$. As in the proof of
  Theorem~\ref{t:1}, we consider the new Priestley space $X' = K_{2n}$
  and its clopen subset $V' = X' \cap V$.
  Then, setting $G_i' = X' \cap G_i$ for each $i\in \{1, \dots, p\}$,
  we obtain a decreasing sequence of clopen subsets of $X'$ that form
  a difference chain for $V'$. However, notice that by the induction
  hypothesis we have
  
  \[
    \bigcup_{i = 1}^{n} (G'_{2i-1}-G'_{2i}) = \left(\bigcup_{i =
        1}^{n} (G_{2i-1}-G_{2i}) \right)\cap K_{2n} \subseteq
    \left(\bigcup_{i = 1}^n (K_{2i-1} - K_{2i}) \right)\cap
    K_{2n}=\emptyset
  \]  
  
  \noindent so that the first $2n$ sets do not contribute to the
  writing of $V'$ as a difference.  It follows that the sequence
  $G'_{2n+1} \supseteq G'_{2n+2} \supseteq \dots \supseteq G'_{2p}$ is
  also a difference chain of closed upsets of $X'$ for $V'$.  Now
  applying Corollary~\ref{c:2} to this sequence, we see that
  
  \[
    K_{2n+1}={\uparrow}(K_{2n}\cap V)={\uparrow} V'\subseteq
    G'_{2n+1}\subseteq G_{2n+1}
  \]
 
  \noindent and
 
  \[
    K_{2n+2}={\uparrow}(K_{2n+1}- V)={\uparrow}({\uparrow}
    V'-V')\subseteq G'_{2n+2} \subseteq G_{2n+2}.
  \]
  
  \noindent We also obtain that $(G_{2n+1}-G_{2n+2})\cap
  K_{2n}=G'_{2n+1}-G'_{2n+2} \subseteq K_{2n+1}- K_{2n+2}$.  On the
  other hand, by Theorem~\ref{t:1}, we have

  \[
    (G_{2n+1}-G_{2n+2})- K_{2n}\subseteq V- K_{2n}=\bigcup_{i = 1}^n
    (K_{2i-1} - K_{2i}).
  \]

  \noindent We thus conclude that

  \[
    \bigcup_{i = 1}^{n+1} (G_{2i-1}-G_{2i}) \ \subseteq\ \bigcup_{i =
      1}^{n+1} (K_{2i-1} - K_{2i})
  \]

  \noindent as required for the inductive step.
\end{proof}

We are now ready to derive the normal form for elements of $D^-$
whenever $D$ is a co-Heyting algebra. This is a straightforward
consequence of Theorem~\ref{thm:co-esak}, Theorem~\ref{t:1} and
Proposition~\ref{p:5}.

\begin{corollary}\label{c:1}
  Let $D$ be a co-Heyting algebra and $b\in D^-$. Define a sequence of
  elements in $ D$ (recall Proposition~\ref{prop:ceiling}) as follows:

  \[
    a_1 = \lceil b \rceil, \qquad a_{2i} = \lceil a_{2i-1} - b \rceil,
    \qquad \text{and} \qquad a_{2i+1} = \lceil a_{2i} \wedge b \rceil,
  \]

  \noindent for $i \ge 1$.  Then, the sequence $\{a_i\}_{i \ge 0}$ is
  decreasing, and there exists $m \ge 1$ such that $a_{2m + 1} = 0$
  and
  
  \begin{equation}
    b = a_1-(a_2-(\dots (a_{2m-1}-a_{2m})
    \hspace*{-1pt}.\hspace*{-1pt}.\hspace*{-1pt}.\hspace*{-1pt})).
    \label{eq:8}
  \end{equation}
  
  Moreover, for every other writing
  
  \[b = c_1 - (c_2 - (\dots (c_{2p-1} - c_{2p})\dots))\]
  
  \noindent as a difference chain with $c_1\ge \dots \ge c_{2p}$ in
  $D$, we have $p \ge m$, $c_i \ge a_i$ for $i \in \{1, \dots, 2p\}$,
  and for each $n\leq m$ we have $\bigvee_{i = 1}^n (c_{2i-1}-c_{2i})
  \leq \bigvee_{i = 1}^n (a_{2i-1} - a_{2i})$.
\end{corollary}

Notice that every finite distributive lattice is a co-Heyting algebra,
so that the above corollary applies to any finite distributive
lattice. Combined with the fact that every bounded distributive
lattice is the direct limit of its finite sublattices and that the
Booleanization is the direct limit of the Booleanizations of these
finite bounded sublattices, we also have the following.

\begin{corollary}\label{cor:difchains}
  Every Boolean element over any bounded distributive lattice may be
  written as a difference chain of elements of the lattice.
\end{corollary}

As shown by Corollary~\ref{c:1}, the case of a co-Heyting algebra is
particularly simple. In order to be able to apply the ideas of this
section in a broader setting, we need a generalization which is easier
to work with in the pointfree setting afforded by canonical
extensions.
\section{Preliminaries on canonical extensions}\label{sec:canext}
Here we provide the required information on canonical extensions. For
further details, please see~\cite{Gehrke14} or~\cite{GehrkeJonsson04}.
\paragraph{Canonical extensions}

Let $D$ be a bounded distributive lattice and $X$ its dual Priestley
space. Then, Priestley duality implies that the Stone map

\[
  D\ \longrightarrow\ \mathcal U(X,\leq),\quad a\ \mapsto\ \widehat{a}
\]

\noindent is an embedding of $D$ into the complete lattice $\mathcal
U(X,\leq)$ of upsets of the poset underlying $X$.  An embedding into a
complete lattice is called a \emph{completion}, and canonical
extension, first introduced by J\' onsson and Tarski
\cite{JonssonTarski51}, comes about from the fact that the above
completion can be uniquely characterized in abstract terms among all
the completions of $D$. Indeed, it is the unique completion
$e:D\hookrightarrow C$ (up to isomorphism) satisfying the following
two properties:

\begin{enumerate}
\item[] \hspace*{-.5cm}{\bf(dense)} Each element of $C$ is a join of
  meets and a meet of joins of elements in the image of $D$;
\item[]\hspace*{-1.1cm} {\bf(compact)} For $S,T\!\subseteq\!D$ with
  $\bigwedge\!e(S)\!\leq\!\bigvee\!e(T)$ in $C$, there are finite
  $S'\!\subseteq\!S$ and $T'\!\subseteq\!T$ with
  $\bigwedge\!e(S')\!\leq\!\bigvee\!e(T')$.
\end{enumerate}

\noindent Thus, instead of working with the the dual space of a
bounded distributive lattice $D$, we will work with its canonical
extension, denoted $D^\delta$. It comes with an embedding
$D\hookrightarrow D^\delta$, which is compact and dense in the above
sense. As stated above this means that $D^\delta$ is isomorphic to
$\mathcal U(X,\leq)$, where~$X$ is the Priestley space of $D$, and
that the embedding $e$ is naturally isomorphic to the Stone map
$a\mapsto \widehat{a}$. In what follows, to lighten the notation, we
will assume (WLOG) that the embedding $e$ is an inclusion so that $D$
sits as a sublattice in $D^\delta$.
\paragraph{Filter and ideal elements}

Since $D$ sits in $D^\delta$ as the clopen upsets sit in $\mathcal
U(X,\leq)$, where $X$ is the Priestley space of $D$, we see that the
join-closure of $D$ in $D^\delta$ will correspond to the lattice of
open upsets of $X$. One can show that these are in one-to-one
correspondence with the ideals of $D$. Thus we denote the join-closure
of $D$ in $D^\delta$ by $I(D^\delta)$ and call the elements of
$I(D^\delta)$ \emph{ideal elements} of $D^\delta$. Similarly the
meet-closure of $D$ in $D^\delta$ will correspond to the lattice of
closed upsets of $X$ and one can show that these are in one-to-one
correspondence with the filters of $D$. Accordingly we denote the
meet-closure of $D$ in $D^\delta$ by $F(D^\delta)$ and call its
elements \emph{filter elements} of $D^\delta$.  Note that, relative to
the concepts of filter and ideal elements, the density property of
$D^\delta$ states that every element of $D^\delta$ is a join of filter
elements and a meet of ideal elements.

We are particularly interested in the closed upsets of $X$, and thus
in $F(D^\delta)$, since they provide canonical difference chains, as
we have just seen in the previous section. We point out another
abstract characterization of $F(D^\delta)$ which will be useful to
us. $F(D^\delta)$ is the free down-directed meet completion of $D$ and
thus it is uniquely determined by the following two properties
\cite[Proposition~2.1]{GePr08}:
\begin{enumerate}
\item[] \hspace*{-.3cm}{\bf(filter dense)} \hspace{.2cm} Each element
  of $F(D^\delta)$ is a down-directed meet of elements from $D$;
\item[]\hspace*{-1cm} {\bf(filter compact)} For
  $S\!\subseteq\!F(D^\delta)$ down-directed and $a\!\in\!D$, if
  $\bigwedge\!S\!\leq\!a$, then there is $s\!\in\!S$ with
  $s\!\leq\!a$.
\end{enumerate}

Notice, that in the particular case of a Boolean algebra $B$, the
order on the dual space is trivial and thus $B^\delta$ is isomorphic
to the full powerset of the dual space $X$ of $B$. Also, the ideal
elements of $B^\delta$ correspond to all the opens of $X$ while the
filter elements of $B^\delta$ correspond to all the closed subsets
of~$X$.
\paragraph{A ceiling function at the level of canonical extensions}
Consider a situation where we have a Boolean algebra $B$ and a bounded
sublattice $D$ of $B$. We saw in Proposition~\ref{prop:ceiling}, that
if $D$ is a co-Heyting algebra and $B$ is its Booleanization, then
there is a lower adjoint $\lceil\ \rceil:B\to D$.  Here we will show,
that \emph{on the level of canonical extensions} any embedding has a
lower adjoint with nice properties for filter elements.

It is a part of the theory of canonical extensions that the embedding
of $D$ in $B$ extends to a complete embedding of the canonical
extensions $D^\delta\hookrightarrow B^\delta$
\cite[Theorem~3.2]{GehrkeJonsson04} which restricts to embeddings for
the filter elements as well as for the ideal elements
\cite[Theorem~2.19]{GehrkeJonsson04}. Since the embedding is a
complete embedding it has both an upper and a lower adjoint, see
Proposition~\ref{prop:adj}.  We will be interested in the lower
adjoint, which we will denote by $\overline{(\ )} \colon B^\delta\to
D^\delta$. Thus we have, for $v\in B^\delta$,

\[
  \overline{v}=\bigwedge\{u\in D^\delta\mid v\leq u\}.
\]

Since $D^\delta\hookrightarrow B^\delta$, we may think of $D^\delta$
as a (complete) sublattice of $B^\delta$ and of the embedding as the
identity. If we do this, then we may see $\overline{(\ )}$ as a
closure operator on $B^\delta$ taking values inside $D^\delta$,
cf. the comment following Proposition~\ref{prop:adj}. We are
particularly interested in the restriction of this closure operator to
the filter elements.

\begin{proposition}\label{prop:overline}
  Let $B$ be a Boolean algebra and $D$ a bounded sublattice of $B$,
  and let $\overline{(\ )}\colon B^\delta\to D^\delta$ be the lower
  adjoint of the embedding $D^\delta\hookrightarrow B^\delta$. Then
  the following properties hold:

  \begin{enumerate}
  \item\label{item:3} for each $u\in B^\delta$, $\overline{u}$ is the
    least element of $D^\delta$ which lies above $u$;
  \item\label{item:4} the map $\overline{(\ )}\colon B^\delta\to
    D^\delta$ sends filter elements to filter elements;
  \item\label{item:5} the map $\overline{(\ )}\colon F(B^\delta)\to
    F(D^\delta)$ preserves down-directed meets.
  \end{enumerate}
\end{proposition}

\begin{proof}
  
  We think of the embedding $D^\delta\hookrightarrow B^\delta$ as an
  inclusion. Thus \ref{item:3} follows by the definition of adjoints:
  for $v\in D^\delta$ we have $\overline{u}\leq v$ if and only if
  $u\leq v$.  For \ref{item:4}, let $v\in F(B^\delta)$ and let $u\in
  D^\delta$ with $v\leq u$. Order dually to $F(D^\delta)$, the free
  directed join completion of $D$ is given by the subframe of ideal
  elements, $I(D^\delta)$.  By the density property of canonical
  extension, every element of $D^\delta$ is a join of filter elements
  and a meet of ideal elements. For each $y\in I(D^\delta)$ with
  $u\leq y$, we have $v\leq y=\bigvee\{a\in D\mid a\leq y\}$ and thus,
  by compactness, there is $a_y\in D$ with $v\leq a_y\leq y$.  Now we
  get
  
  \[
    v\leq\bigwedge\{a_y\mid y\in I(D^\delta), u\leq
    y\}\leq\bigwedge\{y\mid y\in I(D^\delta), u\leq y\}=u.
  \]
  
  \noindent
  Since $\bigwedge\{a_y\mid y\in I(D^\delta), u\leq y\}\in
  F(D^\delta)\subseteq D^\delta$, we have

  \begin{align*}
    \overline{v}& =\bigwedge\{u\in D^\delta, v\leq
                  u\}
    \\ & =\bigwedge\{\bigwedge\{a_y\mid u\leq y\in I(D^\delta)\}\mid
         u\in D^\delta, v\leq u\}\in F(D^\delta).
  \end{align*}
  
  \noindent For \ref{item:5}, let $S$ be a down-directed subset of
  $F(B^\delta)$ with $v=\bigwedge S$. By \ref{item:4},
  $\overline{v}\in F(D^\delta)$, and thus
  $\overline{v}=\bigwedge\{a\in D\mid v\leq a\}$. Let $a\in D$ with
  $v\leq a$, then $\bigwedge S\leq a$ and by the filter compactness
  property of $F(B^\delta)$, there is $w_a\in S$ with $w_a\leq
  a$. Therefore we have
  
  \[
    \bigwedge\{\overline{w}\mid w\in
    S\}\leq\bigwedge\{\overline{w}_a\mid v\leq a\in
    D\}\leq\bigwedge\{a\mid v\leq a\in D\}=\overline{v}.
  \]

  \noindent
  On the other hand, since $v\leq w$ for each $w\in S$, by
  monotonicity of the closure operator, we also have $\overline{v}\leq
  \bigwedge\{\overline{w}\mid w\in S\}$ and thus the closure operator,
  restricted to $F(B^\delta)$, preserves down-directed meets.
\end{proof}

\begin{remark}\label{rem}
  
  Notice that if $B$ is the Booleanization of $D$, and $X$ is the
  Priestley space of $D$, then $B^\delta\cong\mathcal P(X)$,
  $D^\delta\cong\mathcal U(X)$, and the closure operator $\overline{(\
    )}: B^\delta\to D^\delta$ is, according to
  Proposition~\ref{prop:overline}\ref{item:3} simply the map that
  takes a $S\subseteq X$ to the upset ${\uparrow}S$. Furthermore,
  Proposition~\ref{prop:overline}\ref{item:4} tells us that if
  $K\subseteq X$ is closed then ${\uparrow}K$ is also closed. That is,
  it is the canonical extension incarnation of the second assertion in
  Proposition~\ref{p:2}. We did not prove
  Proposition~\ref{prop:overline}\ref{item:5} in topological terms,
  but we could have. It says that if $\{W_i\}_{i\in I}$ is a
  down-directed family of closed subsets of a Priestley space, then

  \[
    \bigcap_{i\in I} {\uparrow}W_i={\uparrow}(\bigcap_{i\in I} W_i).
  \]
  
  \noindent A statement that is not true in general for down-directed
  families of subsets of a poset.
  
\end{remark}

\section{The difference hierarchy and directed families of
  adjunctions}\label{sec:dir-fam-adj}
Let $B$ be a Boolean algebra, $I$ a directed partially ordered set,
$\{S_i\}_{i\in I}$ an indexed family of meet-semilattices, and
$\{f_i\colon B\leftrightarrows S_i\colon g_i\}_{i\in I}$ an indexed
family of adjunctions satisfying:
\begin{enumerate}[label = (D.\arabic*)]
\item\label{item:D1} $\im(g_i)\subseteq \im(g_j)$ for all $i,j\in I$
  with $i\leq j$;
\item\label{item:D2} $\bigcup_{i\in I}\im(g_i):= D$ is a bounded
  sublattice of $B$.
\end{enumerate}

For $b\in B$ and $i\in I$ we denote by
$\overline{b}^i=g_if_i(b)$. This is the closure operator on $B$
associated with the adjunction $f_i\colon B\leftrightarrows S_i\colon
g_i$. We have the following relationship between these closure
operators and the one given by $D$.

\begin{proposition}\label{prop:comparingoverlines}
  
  Let $B$, $D$ and $\{f_i\colon B\leftrightarrows S_i\colon
  g_i\}_{i\in I}$ be as specified above, then, for each $x\in B$, we
  have:
  
  \[
    \overline{x}=\bigwedge_{i\in I} \overline{x}^i
  \]
  
  \noindent where the meet is taken in $B^\delta$.
\end{proposition}

\begin{proof}
  
  For $x\in B$, since $\overline{(\ )}^i$ is a closure operator, we
  have $x\leq \overline{x}^i\in\im(g_i) \subseteq D\subseteq
  D^\delta$. Also $\overline{x}$ is the least element of $D^\delta$
  above $x$.  It follows that $\overline{x}\leq \overline{x}^i$ for
  each $i\in I$ and thus $\overline{x}\leq \bigwedge_{i\in
    I}\overline{x}^i$. On the other hand, by
  Proposition~\ref{prop:overline}\ref{item:4}, since $x\in B\subseteq
  F(B^\delta)$, we have $\overline{x}\in F(D^\delta)$. That is,
  $\overline{x}=\bigwedge\{a\in D\mid x\leq a\}$. Now let $a\in D$
  with $x\leq a$. Then, since $D=\bigcup_{i\in I}\im(g_i)$, there is
  $j\in I$ with $a\in \im(g_j)$.  Now using the fact that $x\leq a$
  and the monotonicity of $\overline{(\ )}^j$ we obtain
  
  \[
    \bigwedge_{i\in I} \overline{x}^i\leq \overline{x}^j\leq
    \overline{a}^j=a.
  \]
  
  \noindent
  We thus have
  
  \[
    \bigwedge_{i\in I} \overline{x}^i\leq \bigwedge\{a\in D\mid x\leq
    a\}=\overline{x}.\popQED
  \]  
\end{proof}

Now fix $b\in B$ and define sequences $\{k_n\}_{n\geq 1}$ and
$\{c_{n,i}\}_{n\geq 1}$, for each $i\in I$, as follows:

\[
  k_1=\overline{b},\ k_{2n}=\overline{k_{2n-1}-b}, \quad\text{ and
  }\quad k_{2n+1}=\overline{k_{2n}\wedge b}
\]

\[
  c_{1,i}=\overline{b}^i, \ c_{2n,i}=\overline{c_{2n-1,i}-b}^i,
  \quad\text{ and }\quad c_{2n+1,i}=\overline{c_{2n,i}\wedge b}^i
\]

\begin{remark}\label{r:k=K}
  
  Note that if $B=D^-$, then the sequence $\{k_n\}_{n\in\N}$ is
  exactly the canonical extension incarnation of the sequence
  $\{K_n\}_{n\in\N}$ given by $V=\widehat{b}$ of Theorem~\ref{t:1}. We
  can say even more: By the universal property of the Booleanization
  $D^-$ of $D$, any embedding $D\hookrightarrow B$ factors through
  $D^-$. Thus it is the composition of two embeddings
  $D\hookrightarrow D^-$ and $D^-\hookrightarrow B$. The arguments
  given above thus apply to both of these and the canonical extensions
  of both of these embeddings have lower adjoints whose composition is
  the lower adjoint of the composition of the two embeddings. But if
  $b\in D^-$, then applying the lower adjoint of the embedding
  $D^-\hookrightarrow B$ to it, leaves it fixed. Accordingly, for any
  embedding $D\hookrightarrow B$, if $b\in B$ is in fact in $D^-$,
  then the sequence $\{k_n\}_{n\in\N}$ is exactly the canonical
  extension incarnation of the sequence $\{K_n\}_{n\in\N}$ given by
  $V=\widehat{b}$ of Theorem~\ref{t:1}.
  
\end{remark}

\begin{lemma}\label{lem:kandc}
  
  The following properties hold for the sequences as defined above:
  
  \begin{enumerate}
  \item\label{item:7} $i\leq j$ implies $k_n\leq c_{n,j}\leq c_{n,i}$
    for all $n\in\N$ and $i,j\in I$;
  \item\label{item:8} $k_n=\bigwedge_{i\in I} c_{n,i}$ for all
    $n\in\N$.
  \end{enumerate}
  
\end{lemma}

\begin{proof}
  
  Define $k_0=c_{0,i}=1$ for all $i\in I$. Also, define $b_n=b$ for
  $n$ odd and $b_n=\neg b$ for $n$ even then we have, for all $n\geq
  1$, $k_{n+1}=\overline{k_n\wedge b_n}$ and similarly for the
  $c_{n,i}$. Proceeding by induction on~$n$, we suppose~\ref{item:7}
  holds for $n\in\N$ and that $i\leq j$. Note that since
  $\im(g_i)\subseteq\im(g_j)\subseteq D$, we have
  $\overline{x}\leq\overline{x}^j\leq \overline{x}^i$ for all $x\in
  B$. Also, by the induction hypothesis $k_n\leq c_{n,j}\leq c_{n,i}$,
  and thus we have
  
  \[
    \overline{k_n\wedge b_n}\leq\overline{c_{n,j}\wedge b_n}\leq
    \overline{c_{n,j}\wedge b_n}^j\leq \overline{c_{n,j}\wedge b_n}^i
    \leq \overline{c_{n,i}\wedge b_n}^i.
  \]
  
  \noindent That is, $k_{n+1}\leq c_{n+1,j}\leq c_{n+1,i}$ as
  required.

  For~\ref{item:8}, again the case $n=0$ is clear by definition and we
  suppose $k_n=\bigwedge_{i\in I} c_{n,i}$. Then we have
  
  \[
    k_{n+1}=\overline{k_n\wedge b_n}=\overline{(\bigwedge_{i\in I}
      c_{n,i})\wedge b_n}=\overline{\bigwedge_{i\in I} (c_{n,i}\wedge
      b_n)}.
  \]
  
  \noindent Now applying Proposition~\ref{prop:overline}\ref{item:5}
  and then Proposition~\ref{prop:comparingoverlines}, we obtain
  
  \[
    k_{n+1}=\bigwedge_{i\in I} \overline{c_{n,i}\wedge
      b_n}=\bigwedge_{i\in I}\bigwedge_{j\in I}
    \overline{c_{n,i}\wedge b_n}^j.
  \]
  
  \noindent Now given $i,j\in I$, since $I$ is directed, there is
  $k\in I$ with $i,j\leq k$. By Lemma~\ref{lem:kandc}\ref{item:7} we
  have $c_{n,k}\leq c_{n,i}$. Combining this with the fact that
  $\im(g_j)\subseteq\im (g_k)$ we obtain
  
  \[
    \overline{c_{n,k}\wedge b_n}^k\leq \overline{c_{n,i}\wedge
      b_n}^k\leq \overline{c_{n,i}\wedge b_n}^j
  \]
  
  \noindent and thus
  
  \[
    k_{n+1}=\bigwedge_{(i,j)\in I^2} \overline{c_{n,i}\wedge
      b_n}^j=\bigwedge_{k\in I} \overline{c_{n,k}\wedge
      b_n}^k=\bigwedge_{k\in I}c_{n+1,k}.\popQED
  \]
  
\end{proof}

\noindent We are now ready to state and prove our main theorem.

\begin{theorem}\label{thrm:main}
  
  Let $B$ be a Boolean algebra, $I$ a directed partially ordered set,
  $\{S_i\}_{i\in I}$ an indexed family of meet-semilattices, and
  $\{f_i\colon B\leftrightarrows S_i\colon g_i\}_{i\in I}$ an indexed
  family of adjunctions satisfying~\ref{item:D1} and~\ref{item:D2},
  that is:
  
  \begin{enumerate}[label = (D.\arabic*)]
  \item\label{item:1} $\im(g_i)\subseteq \im(g_j)$ for all $i,j\in I$
    with $i\leq j$;
  \item\label{item:2} $\bigcup_{i\in I}\im(g_i):= D$ is a bounded
    sublattice of $B$.
  \end{enumerate}

  \noindent For each $b\in B$, define

  \[
    k_1=\overline{b},\quad k_{2n}=\overline{k_{2n-1}-b}, \quad\text{
      and }\quad k_{2n+1}=\overline{k_{2n}\wedge b}
  \]

  \[
    c_{1,i}=\overline{b}^i, \quad c_{2n,i}=\overline{c_{2n-1,i}-b}^i,
    \quad\text{ and }\quad c_{2n+1,i}=\overline{c_{2n,i}\wedge b}^i
  \]

  If $b\in D^-\subseteq B$, then, there is $m\in\N$ and an $i\in I$ so
  that, for each $j\in I$ with $i\leq j$ we have

  \begin{align*}
    b&=k_1-(k_2-\ldots (k_{2m-1}-k_{2m}).\,\!.\,\!.)=\bigvee_{l=1}^m (k_{2l-1}-k_{2l})\\
     &=c_{1,j}-(c_{2,j}-\ldots (c_{2m-1,j}-c_{2m}).\,\!.\,\!.)=\bigvee_{l=1}^m (c_{2l-1,j}-c_{2l,j}).
  \end{align*}

\end{theorem}

Note that for $b\in D^-$ the fact that the first line of the
conclusion holds is precisely the canonical extension reformulation of
Theorem~\ref{t:1}. See also Remark~\ref{r:k=K}.  The fact that the
second line holds is the content of the following lemma.

\begin{lemma}
  
  Let $b, b'\in B$ and $v \in B^\delta$ be such that $v\wedge
  k_{2l+1}=0$ and $b'\leq k_{2l+2}$. Suppose
  $b=v\vee(k_{2l+1}-k_{2l+2})\vee b'$. Then there is an $i\in I$ so
  that, for each $j\in I$ with $i\leq j$ we have
  $b=v\vee(c_{2l+1,j}-c_{2l+2,j})\vee b'$.
  
\end{lemma}

\begin{proof}
  
  Since both $v$ and $k_{2l+1}-k_{2l+2}$ are below $\neg k_{2l+2}$ we
  have $b\leq \neg k_{2l+2}\vee b'$, or equivalently, $b\wedge
  k_{2l+2}\leq b'$. Now by Lemma~\ref{lem:kandc}\ref{item:8} we have
  $b\wedge \bigwedge_{i\in I} c_{2l+2,i}\leq b'$ and by compactness
  there is an $i_1\in I$ so that for all $j\in I$ with $i_1\leq j$ we
  have $b\wedge c_{2l+2,j}\leq b'$, or equivalently, $b\leq \neg
  c_{2l+2,j}\vee b'$. Now, for each $j\in I$ with $i_1\leq j$
  
  \begin{align*}
    b &=(\neg k_{2l+1}\wedge b)\vee(k_{2l+1}\wedge b)= v\vee(k_{2l+1}\wedge b)\\
      &\leq v\vee(k_{2l+1}\wedge (\neg c_{2l+2,j}\vee
        b'))=v\vee(k_{2l+1}\wedge\neg c_{2l+2,j})\vee (k_{2l+1}\wedge
        b')\\ 
      &\leq v\vee(k_{2l+1}- c_{2l+2,j})\vee b'\\
      &\leq v\vee(k_{2l+1}- k_{2l+2})\vee b'=b.
  \end{align*}
  
  \noindent Consequently, for each $j\in I$ with $i_1\leq j$ we have
  $b=v\vee(k_{2l+1}- c_{2l+2,j})\vee b'$. Now, since $c_{2l+2,j} =
  \overline{c_{2l+1,j}-b}^j \ge c_{2l+1,j}-b$, and thus, $b \ge
  c_{2l+1,j}- c_{2l+2,j}$, using also the inequality $c_{2l+1,j} \ge
  k_{2l+1}$ given by Lemma~\ref{lem:kandc}\ref{item:7}, we may deduce
  
  \begin{align*}
    b
    & = v \vee b \vee b' \ge v \vee (c_{2l+1,j}- c_{2l+2,j}) \vee b'
    \\ & \ge  v \vee (k_{2l+1,j}- c_{2l+2,j}) \vee b' = b.
  \end{align*}
  
  \noindent It then follows that for all $j\in I$ with $j \ge i_1$ we
  have
  
  \[
    b=v\vee( c_{2l+1,j}- c_{2l+2,j})\vee b'. \popQED
  \]
\end{proof}

\begin{remark}
  
  Notice that Corollary~\ref{cor:difchains} could also be seen as a
  consequence of Theorem~\ref{thrm:main}.  Let $D$ be any bounded
  distributive lattice, $B$ its Booleanization. For each finite
  bounded sublattice $D'$ of $D$, the embedding $f': D'\hookrightarrow
  D\hookrightarrow B$ has an upper adjoint $g':B\to D'$ given by
  $g'(b)=\bigwedge\{a\in D'\mid b\leq a\}=\min\{a\in D'\mid b\leq a\}$
  and this is a directed collection of adjunctions to which
  Theorem~\ref{thrm:main} applies thus yielding
  Corollary~\ref{cor:difchains}. In fact, in this way, we obtain more
  information as we see that the minimum length chain in $D$ is equal
  to the minimum length chain in $F(D^\delta)$, or equivalently, in
  the lattice of closed upsets of the dual space of $D$. In turn, this
  is the same as the maximum length of difference chains in the dual
  with respect to the clopen corresponding to the given element.
  
\end{remark}

In Section~\ref{sec:LoW} we will give an application of the following
consequence of Theorem~\ref{thrm:main}, which needs its full
generality.

\begin{corollary}\label{c:4}
  
  Let $B$ be a Boolean algebra, $I$ a directed partially ordered set,
  $\{S_i\}_{i \in I}$ a family of meet-semilattices and $\{f_i\colon
  B\leftrightarrows S_i\colon g_i\}_{i \in I}$ an indexed family of
  adjunctions satisfying conditions~\ref{item:1} and~\ref{item:2}.  We
  denote by $D$ the bounded distributive lattice $\bigcup_{i \in
    I}\im(g_i)$. Let $B' \le B$ be a Boolean subalgebra closed under
  each of the closure operators $\overline{(\ )}^i = g_if_i$ for $i\in
  I$.  Then,
  
  \[
    (D \cap B')^- = D^- \cap B',
  \]

  \noindent 
  where we view the Booleanization of any sublattice of $B$ as the
  Boolean subalgebra of $B$ that it generates.
  
\end{corollary}

\begin{proof}
  
  Since $D\cap B'$ is contained in both of the Boolean algebras $D^-$
  (also viewed as a subalgebra of $B$) and $B'$, the Booleanization of
  $D \cap B'$ is contained in their intersection.
  
  For the converse, let $b \in D^- \cap B'$. By
  Theorem~\ref{thrm:main}, there exists an index~$j$ so that $b$ can
  be written as a difference chain
  
  \[
    b = c_{1, j}- (c_{2,j}-(\dots - (c_{2m-1,j}-c_{2m,j})\dots)),
  \]
  
  \noindent where $c_{1,j} = \overline b^j$, $c_{2n, j} =
  \overline{c_{2n-1,j}-b}^j$ and $c_{2n+1,j} =
  \overline{c_{2n,j}\wedge b}^j$, for $n \ge 1$. But then, by
  hypothesis it follows that $c_{1,j}\ge \dots\ge c_{2m,j}$ is a chain
  in $g_jf_j[B'] \subseteq D\cap B'$. Thus, $b$ belongs to $(D \cap
  B')^-$.
\end{proof}

\begin{remark}\label{r:1}
  
  We remark that the closure of $B'$ under the closure operators
  $\overline{(\ )}^i $ for $i\in I$ implies that there is a family of
  adjunctions $\{f'_i\colon B'\leftrightarrows S'_i=
  g_i^{-1}(B')\colon g'_i\}_{i \in I}$ obtained by considering the
  restrictions $f_i'$ and $g_i'$ of $f_i$ and $g_i$,
  respectively. Notice that the closure of $B'$ under $g_if_i$ implies
  that $f_i$ maps $B'$ into $S'_i$ as it is defined. Also, since upper
  adjoints preserve meets, $S_i'$ is indeed a
  meet-semilattice. Finally, $D' = \bigcup_{i \in I} \im(g_i')$ is the
  bounded distributive lattice $D \cap B'$. Indeed, by definition of
  $S'_i$ we have $\im(g_i')=\im(g_i)\cap B'$ and thus

  \[
    D'= \bigcup_{i \in I} \im(g_i')= \bigcup_{i \in I} (\im(g_i)\cap
    B') = \left(\bigcup_{i \in I}\im(g_i)\right)\cap B'=D\cap B'.
  \]

  \noindent 
  Notice that we also have $\im(g_i) \cap B' = g_if_i[B']$. The
  right-to-left inclusion is trivial. Conversely, let $b \in \im(g_i)
  \cap B'$, say $b = g_i(s)$ for some $s \in S_i$. Then, we have
  $g_if_i(b) = g_if_ig_i(s) = g_i(s) = b$, where the second equality
  is well-known to hold for every adjoint pair $(f_i,g_i)$. Therefore,
  it follows that
  
  \[
    D' = D \cap B' = \bigcup_{i \in I}g_if_i[B'].
  \]
  
\end{remark}

We give an example to show that the conclusion of Corollary~\ref{c:4}
is by no means true in general.

\begin{example}
  
  Let $B=\cP(\{a,b,c\})$ be the eight element Boolean
  algebra. Further, let $D$ be the bounded sublattice generated by
  $\{a\}$ and $\{a,b\}$ and let $B'$ be the Boolean subalgebra
  generated by $\{b\}$. Then $B$ is, up to isomorphism, the
  Booleanization of $D$, and thus $D^- \cap B'=B'$, whereas $D\cap
  B'=(D\cap B')^-$ is the two-element Boolean subalgebra of $B$.
  
\end{example}

In order to formulate the application to the theory of formal
languages, we will need some concepts from logic on words.
\section{Preliminaries on formal languages and logic on
  words}\label{sec:back-LoW}
\paragraph{Formal languages}

An \emph{alphabet} is a finite set $A$, a \emph{word over $A$} is an
element of the free $A$-generated monoid $A^*$, and a \emph{language}
is a set of words over some alphabet.
For a word $w \in A^*$, we use $\card w$ to denote the \emph{length}
of $w$, that is, if $w = a_1 \dots a_n$ with each $a_i\in A$, then we
have $\card w = n$.
Given a homomorphism $f: A^* \to M$ into a finite monoid $M$, we say
that a language $L \subseteq A^*$ is \emph{recognized by $f$} if and
only if there is a subset $P \subseteq M$ such that $L = f^{-1}(P)$,
or equivalently, if $L = f^{-1}(f[L])$. The language $L$ is
\emph{recognized by a finite monoid~$M$} provided there is a
homomorphism into $M$ recognizing $L$.
Finally, a language is said to be \emph{regular} if it is recognized
by some finite monoid. Notice that the set of all regular languages
forms a Boolean algebra. Indeed, if a language is recognized by a
given finite monoid then so is its complement, and if $L_1$ and $L_2$
are recognized, respectively, by $M_1$ and $M_2$, then $L_1 \cap L_2$
and $L_1 \cup L_2$ are both recognized by the Cartesian product $M_1
\times M_2$.

We present a technical result that will be used in
Section~\ref{sec:LoW}. Given a monoid~$M$, its powerset $\cP(M)$ is a
monoid when equipped with pointwise multiplication, that is, for
subsets $P_1, P_2 \subseteq M$, we define

\[
  P_1 \cdot P_2 = \{m_1m_2 \mid m_1 \in P_1,\ m_2 \in P_2\}.
\]

\noindent It is easy to see that the preimage of a language recognized
by $M$ under a homomorphism between free monoids is again recognized
by $M$. This is not the case for direct images. However, the forward
image under a \emph{length-preserving homomorphism} between free
monoids (i.e., a homomorphism mapping letters to letters) of a
language recognized by $M$ is recognized by $\cP(M)$
\cite[Theorem~2.2]{Straubing79}. This is an instance of the fact that
modal algebras are dual to co-algebras for the Vietoris monad,
enriched in the category of monoids, see~\cite{BorlidoGehrke18}.
However, since this finitary instance is quite simple to derive, for
completeness, we include a proof.  For a subset $P \subseteq M$, we
denote

\[
  \diamond P = \{Q \subseteq M \mid Q \cap P \neq \emptyset\}.
\]

\begin{lemma}\label{l:6}
  
  Let $f:A^* \to B^*$ and $g:A^* \to M$ be homomorphisms, with~$f$
  length-preserving. Then, the map $h:B^* \to \cP(M)$ defined by $h(w)
  = g[f^{-1}(w)]$ is also a homomorphism. Moreover, for every $P
  \subseteq M$, the equality $h^{-1}(\diamond P) = f[g^{-1}(P)]$
  holds. In particular, if $L \subseteq A^*$ is a language recognized
  by $M$, then $f[L] \subseteq B^*$ is recognized by $\cP(M)$.
  
\end{lemma}

\begin{proof}
  
  We first show that $h$ is a homomorphism. Let $v, w \in B^*$. Then,
  by definition, an element $m$ belongs to $h(vw)$ if and only if
  there exists $u \in f^{-1}(vw)$ such that $m = g(u)$. Since $f$ is
  length-preserving, there is a unique factorization of $u$, say $u =
  u_1u_2$, satisfying $u_1 \in f^{-1}(v)$ and $u_2 \in f^{-1}(w)$.
  Therefore, $h(uv) \subseteq h(u)h(v)$. The converse inclusion is
  trivial (in fact, it holds for every homomorphism $f$ between any
  two monoids).

  Now, let $P \subseteq M$ and $w \in B^*$. We may deduce the
  following:
  
  \begin{align*}
    w \in h^{-1}(\diamond P)
    & \iff g[f^{-1}(w)] \cap P \neq \emptyset
    \\ & \iff \exists u \in A^*
         \colon \ f(u) = w \quad \text{ and } \quad  g(u) \in P
    \\ & \iff w \in f[g^{-1}(P)].
  \end{align*}

  \noindent
  Thus, $h^{-1}(\diamond P) = g[f^{-1}(P)]$ as claimed.
\end{proof}

\begin{corollary}\label{cor:forward-image-reg}

  Let $A$ and $B$ be alphabets and $f:A^* \to B^*$ a length-preserving
  homomorphism. Then the forward image under $f$ of a regular language
  over $A$ is a regular language over $B$.
  
\end{corollary}

We are mostly interested in languages that are defined by first-order
formulas of logic on words. Accordingly, we now introduce this logic.
\paragraph{Syntax of first-order logic on words}

Fix an alphabet $A$. We denote \emph{first-order variables} by $x, y,
z, x_1, x_2, \dots$. \emph{First-order formulas} are inductively built
as follows. For each letter $a \in A$, we consider a \emph{letter
  predicate}, also denoted by $a$, which is unary. Thus, for any
variable $x$, $a(x)$ is an (atomic) formula.  A \emph{$k$-ary
  numerical predicate} is a function $R: \N \to \cP(\N^k)$ satisfying
$R(n) \subseteq \{1, \ldots, n\}^k$ for every $n \in \N$. That is, $R$
is an element of the Boolean algebra $\Pi_{n\in\N}\cP(\{1, \ldots,
n\}^k)$. When we fix a set $\cR$ of numerical predicates, we will
assume it forms a Boolean subalgebra of $\Pi_{n\in\N}\cP(\{1, \ldots,
n\}^k)$.  Each $k$-ary numerical predicate~$R$ and any sequence $x_1,
\dots,x_k$ of first-order variables define an \emph{(atomic) formula}
$R(x_1, \ldots, x_k)$. Finally, Boolean combinations of formulas are
formulas, and if $\varphi$ is a formula and $x_1, \ldots, x_k$ are
distinct variables, then $\forall x_1, \ldots, x_k \ \varphi$ is a
\emph{formula}. In order to simplify the notation, we usually also
consider the quantifier $\exists$: the formula $\exists x_1, \dots,
x_k \ \varphi$ is an abbreviation for $\neg \forall x_1, \dots, x_k \
\neg \varphi$. As usual in logic, we say that a variable occurs
\emph{freely} in a formula provided it is not in the scope of a
quantifier, and a formula is said to be a \emph{sentence} provided it
has no free variables.
\emph{Quantifier-free} formulas are those that are Boolean
combinations of atomic formulas.

\paragraph{Semantics of first-order logic on words}

Let us fix an alphabet $A$ and a set of numerical predicates $\cR$.
To each non-empty word $w = a_1 \dots a_n \in A^*$ with $a_i\in A$, we
associate the relational structure $\cM_w = (\{1, \dots, n\}, A \cup
\cR)$, given by the interpretation $a^w = \{i \in \{1, \dots, n\}\mid
a_i = a\}$, for each $a \in A$, and $R^w = R(n)$, for each $R \in
\cR$.
Models of first-order sentences are words, while models of formulas
with free variables are the so-called \emph{structures}. For a list of
distinct variables ${\bf x} = (x_1, \dots, x_k)$, an
\emph{$\x$-structure} is a map $\{\x\} \to \cM_w$ for some word $w \in
A^*$, where $\{{\x}\}$ denotes the underlying set of $\x$.
We identify maps from $\{\x\}$ to $\{1, \dots, \card w\}$ with
$k$-tuples ${\bf i} = (i_1, \dots, i_k) \in \{1, \dots, \card
w\}^k$. With a slight abuse of notation we write $\{1, \dots, \card
w\}^\x$ for the set of all such maps.
Given a word $w \in A^*$ and a vector ${\bf i} = (i_1, \dots, i_k)\in
\{1, \dots, \card w\}^\x$, we denote by $w_{\x = {\bf i}}$ the
$\x$-structure mapping $x_j$ to $i_j$, for $j = 1, \dots, k$.
Moreover, if $\x = (x_1, \dots, x_k)$ and $\y= (y_1, \dots, y_\ell)$
are disjoint lists of variables, ${\bf i} = (i_1, \dots, i_k) \in \{1,
\dots, \card w\}^\x$ and ${\bf j} = (j_1, \dots, j_\ell) \in \{1,
\dots, \card w\}^{\y}$, then $w_{\x = {\bf i}, \y = {\bf j}}$ denotes
the ${\bf z}$-structure $w_{{\bf z} = {\bf k}}$, where ${\bf z} =
(x_1, \dots, x_k, y_1, \dots, y_\ell)$ and ${\bf k} = (i_1, \dots,
i_k, j_1, \dots, j_\ell)$.
The set of all $\x$-structures is denoted by $A^* \otimes
\x$. Finally, the set of models $L_{\varphi(\x)}$ of a formula
$\varphi(\x)$ having free variables in~$\{{\bf x}\}$ is defined
classically.  For further details concerning logic on words,
see~\cite[Chapter~II]{straubing94}.
The next example should help in understanding the semantics of
first-order logic on an intuitive level.

\begin{example}
  
  The sentence $\varphi = \exists x, y \ (x < y \wedge a(x)\wedge
  b(y))$ is read: ``there are positions $x$ and $y$ such that $x < y$,
  there is an $a$ at position $x$, and there is a $b$ at position
  $y$''. Thus, $\varphi$ defines the language $A^* a A^* b A^*$.
  
\end{example}

Formulas will always be considered up to semantic equivalence, even if
not explicitly said.
We denote by ${\bf FO}[\cN]$ the set of all first-order sentences with
arbitrary numerical predicates (up to semantic equivalence). For a
set~$\cR$ of numerical predicates, ${\bf FO}[\cR]$ denotes the set of
first-order sentences using numerical predicates from~$\cR$.
Notice that, as a Boolean algebra, ${\bf FO}[\cN]$ is naturally
equipped with a partial order, which in turn may be characterized in
terms of semantic containment: $\varphi \le \psi$ if and only if
$L_{\varphi} \subseteq L_{\psi}$. Finally, for a subset of sentences
$S \subseteq {\bf FO}[\cN]$, we use $\cL(S)$ to denote the set of
languages $\{L_\varphi \mid \varphi \in S\}$. In particular, this
yields an embedding of Boolean algebras ${\bf FO}[\cN] \cong \cL({\bf
  FO}[\cN])\subseteq\cP(A^*)$.

\paragraph{Universal quantifiers as adjoints}

Again, we fix a finite alphabet~$A$ and a set of variables~$\x$. We
consider the projection map given by

\[
  \pi\colon A^*\otimes {\x}\twoheadrightarrow A^*, \quad w_{\x = {\bf
      i}}\mapsto w.
\]

\noindent
This gives rise, via the duality between sets and complete and atomic
Boolean algebras, to the complete embedding of Boolean algebras

\[
  \pi^{-1} = (\ )\otimes \x \colon \cP(A^*)\hookrightarrow
  \cP(A^*\otimes \x), \quad L\mapsto \pi^{-1}(L)=L\otimes{\bf x}
\]

\noindent 
This embedding, being a complete homomorphism between complete
lattices, has an upper adjoint which we may call~$\forall$ (and a
lower adjoint~$\exists$). These are given by

\begin{align*}
  \forall:\cP(A^*\otimes \x)
  & \twoheadrightarrow \cP(A^*)
  \\ K
  & \mapsto\forall K=\max\{L\in\cP(A^*)\mid\pi^{-1}(L)=L\otimes\x\subseteq K\}\\
  & \hspace{12mm} = \{w \in A^* \mid \forall{\bf i} \in \{1, \dots, \card w\}^\x,\ w_{\x = {\bf i}} \in K\}\\
  & \hspace{12mm} =(\pi[K^c])^c
\end{align*}

\noindent
and similarly

\begin{align*}
  \exists:\cP(A^*\otimes \x)
  & \twoheadrightarrow \cP(A^*)\\ 
  K
  & \mapsto \exists K = \min \{L \in \cP(A^*)\mid K\subseteq\pi^{-1}(L) = L\otimes \x\}\\
  & \hspace{12mm} = \{w \in A^* \mid \exists{\bf i} \in \{1, \dots, \card w\}^\x,\ w_{\x = {\bf i}} \in K\}\\
  &\hspace{12mm} =\pi[K]
\end{align*}

As is well-known in categorical logic, $\forall$ and $\exists$ are the
semantic incarnations of the classical universal and existential
quantifiers, respectively.
Explicitly, in the case of the universal quantifier, when $K =
L_{\varphi(\x)}$ is definable by a formula $\varphi(\x)$, we have

\begin{equation}
  \label{eq:17}
  \forall L_{\varphi(\x)}=\{w \in A^* \mid \forall{\bf i} \in \{1, \dots, \card w\}^\x,
  \ w_{\x = {\bf i}} \vDash \varphi(\x)\}=L_{\forall \x \ \varphi(\x)}.
\end{equation}

\paragraph{Recognition of languages of structures}

We fix a list of distinct variables $\x = (x_1, \dots, x_k)$. Then,
$2^\x$ is isomorphic to the powerset $\cP(\x)$.  There is a natural
embedding of the set of all $\x$-structures into the free monoid $(A
\times 2^\x)^*$.  Indeed, to an $\x$-structure $w_{\x = {\bf i}}$,
where ${\bf i} = (i_1, \dots, i_k)$, we may assign the word $(a_1,S_1)
\dots (a_n, S_n)$, where $w = a_1 \dots a_n$ with each $a_i\in A$ and,
for $\ell \in \{1,\dots, n\}$, $S_\ell = \{x_j \in \{\x\} \mid i_j =
\ell\}$.  It is not hard to see that this mapping defines an injection
$A^* \otimes \x \hookrightarrow (A \times 2^\x)^*$. Moreover, an
element $(a_1, S_1) \dots (a_n, S_n)$ of $(A \times 2^\x)^*$
represents an $\x$-structure under this embedding precisely when
$\{S_\ell \mid \ell \in \{1, \dots, n\}, \ S_\ell \neq \emptyset\}$
forms a partition of $\x$. From hereon, we view $A^*\otimes \x$ as a
subset of $(A \times 2^\x)^*$ without further mention.

The identification $A^* \otimes \x \subseteq (A \times 2^\x)^*$
enables us to extend the notion of recognition to languages of
structures as follows. Given $L \subseteq A^* \otimes {\bf x}$, we say
that $L$ is \emph{recognized by a homomorphism} $f:(A \times 2^{\bf
  x})^* \to M$ if and only if there is $P \subseteq M$ such that $L =
f^{-1}(P)$, and we say that $L$ is \emph{recognized by a monoid} $M$
if there is a homomorphism $(A \times 2^\x)^* \to M$ recognizing
$L$. Accordingly, a language of structures is \emph{regular} provided
it is recognized by a finite monoid.

An important observation and a well-known fact in logic on words is
that, as a language over the alphabet $A \times 2^\x$, the set of all
$\x$-structures, $A^* \otimes \x$, is a regular language. To simplify
the notation, take $\x = (x)$ and let $N = \{0,1,n\}$ be the
three-element monoid satisfying $n^2 = 0$. Then, the unique
homomorphism $f: (A \times 2^\x)^* \to N$ satisfying $f(a, \emptyset)
= 1$ and $f(a, \{x\}) = n$ recognizes $A^* \otimes \x$ via $\{n\}$.
The general case is handled in a similar manner. This provides a
shortcut for proving regularity for languages of structures.

\begin{proposition}\label{prop:regstr}
  
  Let $L\subseteq A^* \otimes \x$ be a language of structures. Then
  $L$ is regular if and only if there exists a monoid homomorphism
  $f\colon(A\times 2^\x)^* \to M$ into a finite monoid~$M$ and a
  subset $P \subseteq M$ with $L = f^{-1}(P) \cap (A^* \otimes\x)$.

\end{proposition}

\begin{proof}
  
  If $L$ is regular, then there exists a monoid homomorphism
  $f\colon(A\times 2^\x)^* \to M$ onto a finite monoid~$M$ and a
  subset $P \subseteq M$ with $L = f^{-1}(P)$. Thus, in particular, $L
  = f^{-1}(P) \cap (A^* \otimes\x)$. Conversely, suppose there exists
  a monoid homomorphism $f\colon(A\times 2^\x)^* \to M$ onto a finite
  monoid~$M$ and a subset $P \subseteq M$ with $L = f^{-1}(P) \cap
  (A^* \otimes\x)$. Now, let $g\colon (A \times 2^\x)^*\to N$ be the
  monoid homomorphism onto a finite monoid, which recognizes the set
  of all structures, and let $Q\subseteq N$ be so that $A^* \otimes
  \x=g^{-1}(Q)$. Then the product map $f\times g\colon (A \times
  2^\x)^*\to M\times N$ recognizes $L$ using the subset $P\times Q$.
  And we conclude that $L$ is regular.
\end{proof}

\paragraph{Quantifier-free formulas}

We now provide an algebraic characterization of languages definable by
quantifier-free formulas.
Consider a fixed list of distinct variables $\x = (x_1, \dots,
x_k)$. We want a characterization of the languages of the form
$L_{\varphi(\x)}$ for $\varphi(\x)$ a quantifier-free formula whose
free variables are in $\{\x\}$.  We first provide a set theoretic
characterization of these languages. For this purpose, we say that
$L\subseteq X$ is \emph{set theoretically recognized} by $f\colon X\to
Y$ provided there is a subset $P\subseteq Y$ with $L=f^{-1}(P)$.

For a vector of letters ${\bf a} = (a_1, \dots, a_k) \in A^\x$, we
denote by ${\A}(\x)$ the conjunction $a_1(x_1) \wedge \dots \wedge
a_k(x_k)$.
If $w \in A^*$ is a word and ${\bf i} \in \{1, \dots, \card w\}^\x$,
then we denote by $w({\bf i})$ the unique vector ${\bf a}$ for which
$w_{\x = {\bf i}} \models {\bf a}(\x)$. That is, if $w = a_1 \dots
a_n$ with each $a_i\in A$ and ${\bf i} = (i_1, \dots, i_k)$, then
$w({\bf i}) = (a_{i_1},\dots, a_{i_k})$.

\begin{lemma}
  Let $K \subseteq A^* \otimes \x$. Then, $K$ is given by a
  quantifier-free first-order formula over $\x$ if and only if is it
  set theoretically recognized by the map
  
  \[
    c_A : A^* \otimes \x \to \N^{k+1} \times A^k, \qquad w_{\x = {\bf
        i}} \mapsto (\card w, {\bf i}, w({\bf i})).
  \]
  
\end{lemma}

\begin{proof}
  
  Let $P\subseteq\N^{k+1} \times A^k$. For each $\A\in A^k$ and
  $n\in\N$, let

  \[
    R^{\,\A}(n)=\{\ii\in\N^k\mid i_s\leq n \text{ for each }s\leq
    k\text{ and }(\ii,n,\A)\in P\}.
  \]

  \noindent
  Then $R^{\,\A}$ is a ($k$-ary) numerical predicate for each $\A\in
  A^k$ and it is not difficult to see that
  $c_A^{-1}(P)=L_{\varphi(\x)}$ for

  \[
    \varphi(\x)=\bigvee_{\A\in A^k}\left(\A(\x)\wedge
      R^{\,\A}(\x)\right).
  \]

  \noindent
  On the other hand, for $a\in A$ and $i\in\{1,\dots,k\}$

  \[
    L_{a(x_i)}=c_A^{-1}(\N^{k+1} \times \{\A\in A^k\mid a_i=a\})
  \]

  \noindent
  and for $R\subseteq\N^{m+1}$ an $m$-ary numerical predicate and $\y$
  a list of $m$ (not necessarily distinct) variables with
  $\{\y\}\subseteq\{\x\}$, we have

  \[
    L_{R(\y)}=c_A^{-1} (R'\times A^k)
  \]

  \noindent
  where $R'$ is the $k$-ary numerical predicate given by

  \[
    (i_1,\dots,i_k)\in R'(n) \ \iff\ (j_1,\dots,j_m)\in R(n)
  \]

  \noindent
  where $j_s=i_t$ if and only if $y_s=x_t$.
\end{proof}

Observe that the above proof implies that the quantifier-free formulas
with free variables in~$\x$ form a complete Boolean algebra
(isomorphic to a powerset Boolean algebra) and that each of these
quantifier-free formulas may be written as a finite disjunction of
formulas of the very special form ${\bf a}(\x)\wedge R(\x)$.

Now, to obtain an algebraic characterization, let $\varepsilon \notin
A$ be a new symbol and denote $A_\varepsilon = A \cup
\{\varepsilon\}$. We consider the length-preserving homomorphism
$\Theta_\x: (A\times 2^\x)^* \to (A_\varepsilon \times 2^\x)^*$ given
by

\[
  \Theta_\x(a,S) =
  \begin{cases}
    (a, S), & \text{ if } S \neq \emptyset; \\ (\varepsilon, S), &
    \text{ if } S = \emptyset.
  \end{cases}
\]

\noindent
Notice that, given $\x$-structures $v_{\x = {\bf i}}$ and $w_{\x =
  {\bf j}}$, we have

\begin{equation}
  \label{eq:13}
  \Theta_\x(v_{\x = {\bf i}}) = \Theta_\x(w_{\x = {\bf j}}) \iff \card
  v = \card w,\ {\bf i} = {\bf j}, \text{ and } v({\bf i}) = w({\bf j}).
\end{equation}

\noindent
Using this observation, it is straightforward to show:

\begin{lemma}
  
  The following diagrams are well defined and commute
  
  \begin{center}
    \begin{tikzpicture}[node distance = 20mm]
      \node (A) at (0,0) {$(A \times 2^\x)^*$}; \node[right of = A,
      xshift = 15mm] (B) {$(A_\varepsilon \times 2^\x)^*$};
      \node[below of = A] (C) {$A^* \otimes \x$}; \node[below of = B]
      (D) {$A_\varepsilon^*\otimes \x$}; \node[below of = C] (E)
      {$\N^{k+1} \times A^k$}; \node[below of = D] (F) {$\N^{k+1}
        \times A_\varepsilon^k$};
      \draw[->] (A) to node[above]{$\Theta_\x$} (B); \draw[->] (C) to
      node[above]{$\Theta_\x|_{(\ ) \otimes \x}$} (D); \draw[>->] (E)
      to node[above]{$e$} (F);
      \draw[>->] (C) to (A); \draw[>->] (D) to (B); \draw[->](C) to
      node[left] {$c_A$} (E); \draw[->](D) to node[right]
      {$c_{A_\varepsilon}$} (F);
    \end{tikzpicture}
  \end{center}

  \noindent
  where $e$ is the natural inclusion obtained from the inclusion of
  $A$ in $A_\varepsilon$.  Furthermore,
  $\Theta_{\x}^{-1}(A^*_\varepsilon \otimes \x) = A^* \otimes \x$ and
  the restriction of $c_{A_\varepsilon}$ to $\Theta_{\x}[A^* \otimes
  \x]$ is a bijection onto $e[\N^{k+1} \times A^k]$.
\end{lemma}

As an immediate consequence, we have:

\begin{corollary}\label{cor:quantfree}
  
  Let $L \subseteq A^* \otimes \x$ be a language. Then, the following
  are equivalent:
  
  \begin{enumerate}
  \item\label{item:6} $L$ is definable by a quantifier-free formula;
  \item\label{item:16} $L = \Theta_{\x}^{-1}(\Theta_{\x}[L])$;
  \item\label{item:17} there is a subset $P \subseteq A_\varepsilon^*
    \otimes \x$ such that $L = \Theta_{\x}^{-1}(P)$.
  \end{enumerate}
  
\end{corollary}

\paragraph{The alternation hierarchy}

The so-called \emph{alternation hierarchy} has been widely considered
in the literature and its decidability (beyond $k=2$) remains an open
problem (see~\cite{PlaceZeitoun18}). The hierarchy classifies formulas
according to the minimum number of alternations of quantifiers that is
needed to express them in prenex-normal formula, that is, in the form

\begin{equation}
  \psi = Q_1{\x}_1  \dots Q_k{\x}_k\ \varphi(\x_1, \dots,
  \x_k),\label{eq:16}
\end{equation}

\noindent
where $\varphi$ is a quantifier-free formula, $Q_1, \dots, Q_k \in
\{\forall, \exists\}$ and $Q_\ell = \forall$ if and only if
$Q_{\ell+1} = \exists$ for each $\ell = 1, \dots, k-1$.
It is a well-known fact that for every first-order formula there is a
semantically equivalent one in prenex-normal form.
For $k \ge 1$ and a set of numerical predicates $\cR$, $\Pi_{k}[\cR]$
consists of all the sentences of ${\bf FO}[\cR]$ that are semantically
equivalent to a sentence of the form~\eqref{eq:16} where $Q_1 =
\forall$.
In particular, we have $\Pi_k[\cR] \subseteq \Pi_\ell[\cR]$ whenever
$k \le \ell$.
Similarly, $\Sigma_k[\cR]$ denotes the set of all sentences that are
semantically equivalent to a sentence of the form~\eqref{eq:16} with
$Q_1 = \exists$, and $\Sigma_k[\cR] \subseteq \Sigma_\ell[\cR]$
whenever $k \le \ell$.
It is not hard to see that both $\Pi_k[\cR]$ and $\Sigma_k[\cR]$ are
closed under disjunction and conjunction but not under negation. In
other words, $\Pi_k[\cR]$ and $\Sigma_k[\cR]$ are lattices, but not
Boolean algebras. We denote by $\cB\Pi_k[\cR]$ and by
$\cB\Sigma_k[\cR]$ the Boolean algebras generated by $\Pi_k[\cR]$ and
by $\Sigma_k[\cR]$, respectively, that is,

\[
  \cB\Pi_k[\cR] = (\Pi_k[\cR])^- \qquad\text{and}\qquad
  \cB\Sigma_k[\cR] = (\Sigma_k[\cR])^-.
\]

\noindent Clearly, we have $\cB\Pi_k[\cR] = \cB\Sigma_k[\cR] \subseteq
\Pi_{k+1}[\cR], \Sigma_{k+1}[\cR]$.

In this paper we are only concerned with the first level of this
hierarchy. For notational convenience, we will work with the fragment
$\Pi_1[\cN]$, although everything we prove for $\Pi_1[\cN]$ admits a
dual statement for $\Sigma_1[\cN]$.
We use $\reg$ to denote both the subset of ${\bf FO}[\cN]$ defining
regular languages, and the subset of numerical predicates for which
the associated language of structures $L_{R(\x)}$, where $\x$ is a
list of distinct variables of the same length as the arity of $R$, is
regular over the alphabet $A\times 2^\x$ as discussed under
recognition of languages of structures.
For $k > 1$, it is an open problem to determine whether the following
equality holds:

\[
  \cB\Pi_k[\cN] \cap \reg = \cB\Pi_k[\reg].
\]

\noindent Using the results of Section~\ref{sec:dir-fam-adj}, we
provide an elementary proof of the case $k = 1$, which was first
proved in~\cite{MacielPeladeauTherien00}.

Every formula of $\Pi_1[\cR]$ is of the form $\psi = \forall{\x}\
\varphi(\x)$, for some quantifier-free formula $\varphi(\x)$. Inside
$\Pi_1[\cR]$, we classify formulas according to the size of $\x$: we
let $\Pi_1^d[\cR]$ consist of all equivalence classes of such formulas
for which there is a representative~$\psi$ for which~$\x$ has~$d$
variables.
We remark that, $\Pi_1^d[\cR]$ is closed under conjunction, since the
formulas $\forall \x \ \varphi(\x) \wedge \forall \x \ \psi(\x)$ and
$\forall \x \ (\varphi(\x) \wedge \psi(\x))$ are semantically
equivalent.  However, in general, $\Pi_1^d[\cR]$ fails to be closed
under disjunction.

\begin{example}
  
  Let $\varphi(x) = a(x)$ and $\psi(x) = b(x)$. Then, $\forall x \
  \varphi(x)$ defines the language $a^*$, while $\forall x \ \psi(x)$
  defines the language $b^*$ and thus these are both in
  $\Pi_1^1[\cN]$.  The disjunction $\forall x \ \varphi(x)\vee \forall
  x \ \psi(x)$ defines the language $a^*\cup b^*$, while $\forall x \
  (\varphi(x) \vee \psi(x))$ defines the language $\{a,b\}^*$. Indeed,
  one can show that $\forall x \ \varphi(x)\vee \forall x \ \psi(x)$
  is not in $\Pi_1^1[\cN]$ while it is in $\Pi_1^2[\cN]$ as witnessed
  by the sentence $\forall x_1,x_2 \ \left(\varphi(x_1)\vee
    \psi(x_2)\right)$.
  
\end{example}

\section{An application to Logic on Words}\label{sec:LoW}

In this section we combine Corollary~\ref{c:4} and Remark~\ref{r:1} to
prove the equality

\begin{equation}
  \label{eq:9}
  \cB \Pi_1[\cN] \cap \reg = \cB\Pi_1[\reg].
\end{equation}

\noindent The idea is the following. Combining the fact that universal
quantification may be seen as an adjoint and our algebraic
characterization of quantifier-free formulas we obtain a directed
family of adjunctions on $\cP(A^*)$ with joint image equal to
$\cL(\Pi_1[\cN])$ allowing us to fit into the setting of
Theorem~\ref{thrm:main}. Finally we show that these adjunction
restrict correctly to the regular fragment thus allowing us to apply
Corollary~\ref{c:4} and Remark~\ref{r:1}, thereby concluding
that~\eqref{eq:9} holds.

Universal quantification (as an adjoint) and recognition of
quantifier-free formulas are based on the following two maps,
respectively

\[
  A^*\stackrel{\
    \pi}{\twoheadleftarrow}A^*\otimes\x\stackrel{\Theta_\x\
  }{\longrightarrow}A_\varepsilon^*\otimes\x
\]

\noindent
Dually this gives rise to

\[
  \begin{tikzpicture}
    \node (A) {{$\cP(A^*)$}}; \node (B) [node distance=3.5cm, right
    of=A] { {$\cP(A^*\otimes\x)$}}; \node (C) [node distance=3.5cm,
    right of=B] { {$\cP(A_\varepsilon^*\otimes\x)$}}; \draw[<-,bend
    right] (A) to node [below,midway] {$\exists$} (B); \draw[>->] (A)
    to node [above,midway] {$\pi^{-1}$} (B); \draw[->, bend right] (B)
    to node [above,midway]{$\forall$} (A); \draw[->,bend right] (B) to
    node [below,midway] {$\Theta_\x[\ ]$} (C); \draw[<-] (B) to node
    [above,midway] {$\Theta_\x^{-1}$} (C); \draw[<-, bend right] (C)
    to node [above,midway]{$(\Theta_\x[(\ )^c])^c$} (B);
  \end{tikzpicture}
\]

\noindent In particular, we have a (correct) composition of
adjunctions as follows

\[
  \begin{tikzpicture}
    \node (A) {{$\cP(A^*)$}}; \node (B) [node distance=3.5cm, right
    of=A] { {$\cP(A^*\otimes\x)$}}; \node (C) [node distance=3.5cm,
    right of=B] { {$\cP(A_\varepsilon^*\otimes\x)$}}; \draw[>->] (A)
    to node [above,midway] {$\pi^{-1}$} (B); \draw[->, bend right] (B)
    to node [above,midway]{$\forall$} (A); \draw[->,bend right] (B) to
    node [below,midway] {$\Theta_\x[\ ]$} (C); \draw[<-] (B) to node
    [above,midway] {$\Theta_\x^{-1}$} (C);
  \end{tikzpicture}
\]

\noindent
That is,

\[
  f_k=\Theta_\x[\pi^{-1}(\ )]\colon
  \ce{\cP(A^*)\myleftrightarrows{\rule{.75cm}{0cm}}
    \cP(A_\varepsilon^*\otimes\x)} \colon \forall(\Theta_\x^{-1}(\
  ))=g_k
\]

\noindent
is an adjunction and, combining quantification as adjunction with the
description of quantifier free formulas in
Corollary~\ref{cor:quantfree}\ref{item:17}, we have $L\subseteq A^*$
is in $\cL(\Pi_1^k[\cN])$ if and only if $L=\forall (\Theta_\x^{-1}
(P))=g_k(P)$ for some $P\subseteq A_\varepsilon^*\otimes\x$. That is,
$(f_k,g_k)$ is an adjunction with associated closure operator

\[
  \lceil L\rceil_k:=g_kf_k(L)=\forall
  \Theta_\x^{-1}(\Theta_\x[L\otimes\x]),
\]

\noindent
and $\im(\, \lceil \ \
\rceil_k)=\im(g_k)=\cL(\Pi_1^k[\cN])$. Accordingly, we are in the
situation of Theorem~\ref{thrm:main} with

\[
  B=\cP(A^*)\text{ and
  }D=\bigcup_{k\in\N}\im(g_k)=\bigcup_{k\in\N}\cL(\Pi_1^k[\cN])=\cL(\Pi_1[\cN]).
\]

\noindent
We now aim to apply Corollary~\ref{c:4} with $B'=\cL(\reg)$, the
Boolean algebra of all regular languages over the alphabet $A$. This
is possible given the following fact.

\begin{lemma}\label{lem:final}
  
  For each $k\in\N$, if $L\subseteq A^*$ is regular, then so is
  $\lceil L\rceil_k$.

\end{lemma}

\begin{proof}
  
  Fix $k\in\N$ and suppose $L\subseteq A^*$ is regular. We proceed
  through the four maps whose composition defines $\lceil \ \
  \rceil_k$.

  {\bf Claim 1.} $L\otimes\x$ is regular.\\
  Note that if $\mu\colon A^*\to M$ is a finite monoid recognizing
  $L$, then $\pi^*\colon(A\times 2^\x)^*\to A^*$ composed with $\mu$,
  where $\pi^*$ is the homomorphism extending the projection of
  $A\times 2^\x$ onto $A$, recognizes $L\otimes\x$ once we restrict to
  structures.  By Proposition~\ref{prop:regstr} it follows that
  $L\otimes\x$ is regular.

  {\bf Claim 2.} $\Theta_\x[L\otimes\x]$ is regular.\\
  Here we note that $\Theta_\x\colon(A\times
  2^\x)^*\to(A_\varepsilon\times 2^\x)^*$ is a length-preserving
  homomorphism so that $L\otimes\x$ regular implies
  $\Theta_\x[L\otimes\x]$ is regular by
  Corollary~\ref{cor:forward-image-reg}.

  {\bf Claim 3.} $\Theta_\x^{-1}(\Theta_\x[L\otimes\x])$ is regular.\\
  This is immediate as the inverse image with respect to a
  homomorphism between free monoids of a regular language is always
  regular: If $\Theta_\x[L\otimes\x]$ is recognized by $f'\colon
  (A_\varepsilon\times 2^\x)^*\to M'$ then the composition
  $f'\circ\Theta_\x\colon(A\times 2^\x)^*\to M'$ recognizes
  $\Theta_\x^{-1}(\Theta_\x[L\otimes\x])$.

  {\bf Claim 4.} $\forall(\Theta_\x^{-1}(\Theta_\x[L\otimes\x]))$ is regular.\\
  As observed in Section~\ref{sec:back-LoW}, the upper adjoint
  $\forall$ is given by $K\mapsto (\pi[K^c])^c$ where $\pi\colon
  A^*\otimes\x\to A^*$ is the restriction of $\pi^*\colon(A\times
  2^\x)^*\to A^*$ to structures.  It follows that
  $\pi[K^c]=\pi^*[K^c\cap(A^*\otimes\x)]$. Now, since
  $K=\Theta_\x^{-1}(\Theta_\x[L\otimes\x])$ is regular, $K^c$ is also
  regular and $K^c\cap(A^*\otimes\x)$ is regular. Further, noting that
  $\pi^*$ is a length-preserving monoid homomorphism, it follows by
  Corollary~\ref{cor:forward-image-reg} that $\pi^*[(\Theta_\x^{-1}
  (\Theta_\x[L\otimes\x]))^c\cap A\otimes\x]$ is regular. Finally, we
  conclude that its complement
  $\forall(\Theta_\x^{-1}(\Theta_\x[L\otimes\x]))=(\pi^*[(\Theta_\x^{-1}(\Theta_\x[L\otimes\x]))^c\cap
  A\otimes\x])^c$ is regular as required.
\end{proof}

As a consequence, Corollary~\ref{c:4} applies and we obtain:

\begin{corollary}

  Considering each of the following Booleanizations as subalgebras of
  $\cP(A^*)$, we have
  
  \[
    (\cL(\Pi_1[\cN])\cap\cL(\reg))^-=(\cL(\Pi_1[\cN]))^-\cap\cL(\reg).
  \]
  
\end{corollary}

Finally, applying Remark~\ref{r:1} in this particular case, we see
that

\begin{align*}
  \cL(\Pi_1[\cN]) \cap \cL(\reg)
  & =\bigcup_{k\in\N} g_kf_k[\cL(\reg)].
\end{align*}

\noindent
The languages in $g_kf_k[\cL(\reg)]$ are exactly the languages $\lceil
L\rceil_k$ for $L$ regular.  By the proof of Lemma~\ref{lem:final}, we
have that $\lceil L\rceil_k=\forall (\Theta_\x^{-1}(P))$ where
$P=\Theta_\x[L\otimes\x]\subseteq(A_\varepsilon\times 2^\x)^*$, and,
by Claim 2 in particular, we have that $P$ is regular. That is,
$\lceil L\rceil_k=L_{\forall\x\,\varphi(\x)}$ where the atomic formula
$\varphi(\x)$ is regular, or equivalently, $\lceil
L\rceil_k\in\cL(\Pi_1[\reg])$.

On the other hand if $\varphi(\x)$ is an atomic formula that is
regular, then the arguments in Claims~3 and~4 show that
$L_{\forall\x\,\varphi(\x)}$ is regular. Thus

\[
  g_kf_k[\cL(\reg)]=\cL(\Pi_1^k[\reg])
\]

\noindent
and we have

\begin{align*}
  (\cL(\Pi_1[\reg]))^-
  & =( \bigcup_{k\in\N}\cL(\Pi_1^k[\reg]
    )^-=(\cL(\Pi_1[\cN]) \cap
    \cL(\reg))^-
  \\ & =(\cL(\Pi_1[\cN]))^-\cap\cL(\reg).
\end{align*}

\noindent
Since we consider logical formulas up to semantic equivalence, we
obtain the desired result:

\begin{theorem}
  The following equality holds:
  
 \[
   \cB\Pi_1[\reg]= \cB\Pi_1[\cN] \cap \reg.
 \]
 
\end{theorem}


\begin{thebibliography}{99}
  
\bibitem{BorlidoGehrke18}
C.~{Borlido} and M.~{Gehrke}.
\newblock {A note on powers of Boolean spaces with internal semigroups}.
\newblock {\em ArXiv e-prints}, page arXiv:1811.12339, November 2018.

\bibitem{CartonPerrinPin17}
O.~Carton, D.~Perrin, and J.{-}{\'{E}}. Pin.
\newblock A survey on difference hierarchies of regular languages.
\newblock {\em  Logical Methods in Computer Science} 14(1), 2017.

\bibitem{Chen66}
C.~C. Chen.
\newblock Free {B}oolean extensions of distributive lattices.
\newblock {\em Nanta Math.}, 1:1--14, 1966/1967.

\bibitem{DaveyPriestley02}
B.~A. Davey and H.~A. Priestley.
\newblock {\em Introduction to Lattices and Order, 2nd edition}.
\newblock Cambridge University Press, 2002.

\bibitem{Esakia74}
L.~L. \`Esakia.
\newblock Topological {K}ripke models.
\newblock {\em Dokl. Akad. Nauk SSSR}, 214:298--301, 1974.

\bibitem{Esakia85}
L.~L. \`Esakia.
\newblock {\em Heyting algebras: Duality theory}.
\newblock ``Metsniereba'', Tbilisi, 1985.
\newblock (in Russian).

\bibitem{Gehrke14}
M.~Gehrke.
\newblock Canonical extensions, {E}sakia spaces, and universal models.
\newblock In {\em Leo {E}sakia on duality in modal and intuitionistic logics},
  volume~4 of {\em Outst. Contrib. Log.}, pages 9--41. Springer, Dordrecht,
  2014.

\bibitem{GehrkeJonsson04}
M.~Gehrke and B.~J{\'o}nsson.
\newblock Bounded distributive lattices expansions.
\newblock {\em Mathematica Scandinavica}, 94(2):13--45, 2004.

\bibitem{GePr08}
M.~Gehrke and H.~A. Priestley.
\newblock Canonical extensions and completions of posets and lattices.
\newblock {\em Rep. Math. Logic}, 43:133--152, 2008.

\bibitem{GlaserSchmitzSelivanov16}
C.~Gla{\ss}er, H.~Schmitz, and V.~Selivanov.
\newblock Efficient algorithms for membership in {B}oolean hierarchies of
  regular languages.
\newblock {\em Theoret. Comput. Sci.}, 646:86--108, 2016.

\bibitem{GratzerSchmidt58}
G.~Gr\"atzer and E.~T. Schmidt.
\newblock On the generalized {B}oolean algebra generated by a distributive
  lattice.
\newblock {\em Nederl. Akad. Wetensch. Proc. Ser. A. 61 = Indag. Math.},
  20:547--553, 1958.

\bibitem{Hausdorff57}
F.~Hausdorff.
\newblock {\em Set theory}.
\newblock Chelsea Publishing Company, New York, 1957.
\newblock Translated by John R. Aumann, et al.

\bibitem{JonssonTarski51}
B.~J{\'o}nsson and A.~Tarski.
\newblock Boolean algebras with operators i.
\newblock {\em Amer. J. Math.}, 73:891--939, 1951.

\bibitem{MacielPeladeauTherien00}
A.~Maciel, P.~P\'eladeau, and D.~Th\'erien.
\newblock Programs over semigroups of dot-depth one.
\newblock {\em Theoret. Comput. Sci.}, 245(1):135--148, 2000.
\newblock Semigroups and algebraic engineering (Fukushima, 1997).

\bibitem{MacNeille39}
H.~M. MacNeille.
\newblock Extension of a distributive lattice to a {B}oolean ring.
\newblock {\em Bull. Amer. Math. Soc.}, 45(6):452--455, 1939.

\bibitem{Nerode59}
A.~Nerode.
\newblock Some {S}tone spaces and recursion theory.
\newblock {\em Duke Math. J.}, 26:397--406, 1959.

\bibitem{Peremans57}
W.~Peremans.
\newblock Embedding of a distributive lattice into a {B}oolean algebra.
\newblock {\em Nederl. Akad. Wetensch. Proc. Ser. A. {\bf 60} = Indag. Math.},
  19:73--81, 1957.

\bibitem{PlaceZeitoun18}
T.~Place and M.~Zeitoun.
\newblock Going higher in first-order quantifier alternation hierarchies on
  words.
\newblock {\em Journal of the ACM}, 2018.

\bibitem{Priestley70a}
H.~A. Priestley.
\newblock Representation of distributive lattices by means of ordered stone
  spaces.
\newblock {\em Bull. London Math. Soc.}, 2:186--190, 1970.

\bibitem{Stone37}
M.~H. Stone.
\newblock Applications of the theory of {B}oolean rings to general topology.
\newblock {\em Trans. Amer. Math. Soc.}, 41(3):375--481, 1937.


\bibitem{Straubing79}
H.~Straubing.
\newblock Recognizable sets and power sets of finite semigroups.
\newblock {\em Semigroup Forum}, 18(4):331--340, 1979.


\bibitem{straubing94}
H.~Straubing.
\newblock {\em Finite automata, formal logic, and circuit complexity}.
\newblock Progress in Theoretical Computer Science. Birkh\"auser Boston, Inc.,
  Boston, MA, 1994.

\end{thebibliography}
\end{document}